\documentclass[11pt,reqno]{article}
\usepackage[margin=1in]{geometry}

\usepackage{amsmath}
\usepackage{amssymb}
\usepackage{graphicx}
\usepackage[colorlinks=true, allcolors=blue]{hyperref}
\usepackage[utf8]{inputenc}
\usepackage{verbatim}
\usepackage{caption,subcaption}
\usepackage[overload]{empheq}
\usepackage{cleveref} 
\usepackage{amsthm}
\usepackage{epsf,epsfig,mathtools,hyperref,amsthm,bbm,xcolor,pifont}
\usepackage{algorithm}
\usepackage{algpseudocode}
\usepackage{bm}
\usepackage{dirtytalk}

\def \E{\mathbb{E}}

\def \N{\mathbb{N}}
\def \P{\mathbb{P}}
\def \Q{\mathbb{Q}}
\def \R{\mathbb{R}}
\def \Pc{\mathcal{P}}

\def \Jc{\mathcal{J}}
\def \Lc{\mathcal{L}}
\def \Ic{\mathcal{I}}
\def \Dc{\mathcal{D}}
\def \cD{\mathcal{D}}
\def \cF{\mathcal{F}}
\def \cW{\mathcal{W}}
\def \Wc{\mathcal{W}}
\def \cB{\mathcal{B}}

\newcommand{\Div}{{\operatorname{div}}}
\newcommand{\KL}{D_{\operatorname{KL}}}

\DeclareMathOperator*{\argmin}{arg\,min}
\DeclareMathOperator*{\esssup}{ess\,sup}

\newtheorem{remark}{Remark}[section]
\newtheorem{lemma}{Lemma}[section]
\newtheorem{theorem}{Theorem}[section]
\newtheorem{definition}{Definition}[section]
\newtheorem{proposition}{Proposition}[section]
\newtheorem{corollary}{Corollary}[section]

\numberwithin{equation}{section}

\title{Mean-Field Langevin Diffusions with Density-dependent Temperature\thanks{The authors would like to thank Chii-Ruey Hwang, Yuming Paul Zhang, and Xunyu Zhou for fruitful discussions.}}
\author{Yu-Jui Huang\thanks{
Department of Applied Mathematics, University of Colorado, Boulder, CO 80309-0526, USA, email: \texttt{yujui.huang@colorado.edu}. Partially supported by National Science Foundation (DMS-2109002).
}
\and
Zachariah Malik\thanks{
Department of Applied Mathematics, University of Colorado, Boulder, CO 80309-0526, USA, email: \texttt{zachariah.malik@colorado.edu}.
}
}

\begin{document}
\maketitle

\begin{abstract}
In the context of non-convex optimization, we let the temperature of a Langevin diffusion to depend on the diffusion's own density function. The rationale is that the induced density captures to some extent the landscape imposed by the non-convex function to be minimized, such that a density-dependent temperature provides location-wise random perturbation that may better react to, for instance, the location and depth of local minimizers. As the Langevin dynamics is now self-regulated by its own density, it forms a mean-field stochastic differential equation (SDE) of the Nemytskii type, distinct from the standard McKean-Vlasov equations. 
Relying on Wasserstein subdifferential calculus, we first show that the corresponding (nonlinear) Fokker-Planck equation has a unique solution. Next, a weak solution to the SDE is constructed from the solution to the Fokker-Planck equation, by Trevisan's superposition principle. As time goes to infinity, we further show that the induced density converges to an invariant distribution, which admits an explicit formula in terms of the Lambert $W$ function. A numerical example suggests that the density-dependent temperature can simultaneously improve the accuracy of and rate of convergence to the estimate of global minimizers. 
\end{abstract}

\smallskip


\section{Introduction}\label{sec:intro}
When it comes to minimizing a general (non-convex) $\Psi: \R^d \to \R$, significant developments have long  centered around the Langevin stochastic differential equation (SDE)
\begin{equation}\label{Langevin}
dY_t = -\nabla\Psi(Y_t)\, dt+\sqrt{2\lambda}\, dB_t,\quad Y_0 = y\in\R^d, 
\end{equation}
where $\lambda>0$ is given and $B$ is a $d$-dimensional Brownian motion. The intuition behind is straightforward: as gradient descent ``$-\nabla\Psi(Y_t)\, dt$'' alone will trap $Y_t$ at local minimizers of $\Psi$, random perturbation ``$\sqrt{2\lambda}\, dB_t$'' is added to help $Y_t$ escape from local minimizers. 
Early studies reveal that \eqref{Langevin} has a unique limiting distribution (as $t\to\infty$), which increasingly concentrates on global minimizers of $\Psi$ when $\lambda>0$ (called the {\it temperature} parameter) tends to zero; see \cite{GH86, CHS87, Hwang80}. 



Ideally, the temperature $\lambda$ should evolve over time to control the magnitude of random perturbation strategically. This has been implemented along two perspectives. The first one takes $\lambda=\lambda(t)$, a non-increasing deterministic function in time, following the intuition that less random perturbation is needed over time (as $Y_t$ should gradually move near global minimizers of $\Psi$). The $\lambda(t)$ is usually exogenously given; see e.g., \cite{GH86, CHS87, HKS89, GM91, Marquez97}. 
An exception is \cite{MN01}, where desired long-time behavior of $Y_t$ is imposed to inversely identify an ordinary differential equation for $\lambda(t)$. The second perspective takes $\lambda$ to be a function on $\R^d$, so that $\lambda=\lambda(Y_t)$ in \eqref{Langevin}; namely, the temperature evolves {\it stochastically}, depending on the present state $Y_t$. This can potentially better react to the actual landscape imposed by $\Psi$, as recently argued in \cite[Section 1]{GXZ22}: 
\begin{quotation}
{\it ``For instance, in order to quickly escape a deeper local minimum, one should impose a higher temperature on the process. On the other hand, only lower temperatures are needed when the current process is not trapped in local minima...''}
\end{quotation}
State-dependent temperature has not been studied much. It is exogenously given as $\lambda=f(\Psi)$ in \cite{FQG97}, for an increasing $f\ge 0$.
The recent work  \cite{GXZ22} endogenously designs $y\mapsto\lambda(y)$, by solving a related entropy-regularized stochastic control problem (as a proxy for $\min_{y\in\R^d}\Psi(y)$). 

This paper proposes a new perspective---taking the temperature $\lambda$ to be {\it distribution-dependent}. 
Let $\rho^{Y_t}$ denote the probability density function of $Y_t$ for all $t\ge0$. 
We propose to consider 
\begin{equation}\label{Langevin distri.}
dY_t = -\nabla\Psi(Y_t)\, dt+f\left(\rho^{Y_t}(Y_t)\right)\, dB_t,\quad \rho^{Y_0} = \rho_0, 
\end{equation}
where $\rho_0$ is an initial density and $f:[0,\infty)\to [0,\infty)$ is an increasing function of choice. Imagine that at time 0, numerous points $y\in\R^d$ are sampled from $\rho_0$. As time passes by, they are drawn closer to local minimizers of $\Psi$ by gradient descent ``$-\nabla\Psi(Y_t)\, dt$'', such that the density $\rho^{Y_t}$ tends to be large near local minimizers and small otherwise. Hence, for $Y_t = y\in\R^d$ near a local minimizer, with $\rho^{Y_t}(y)$ large and $f$ increasing, ``$f(\rho^{Y_t}(y))\, dB_t$'' can perturb $y$ very forcefully, helping it escape  effectively from the local minimizer; for $Y_t=y\in\R^d$ far away from local minimizers, with $\rho^{Y_t}(y)$ small and $f$ increasing, ``$f(\rho^{Y_t}(y))\, dB_t$'' may perturb $y$ only slightly, letting ``$-\nabla\Psi(y)dt$'' take command. In other words, the density $y\mapsto\rho^{Y_t}(y)$ naturally gives a sketch of the landscape (i.e., the location and depth of local minimizers), so that \eqref{Langevin distri.} may exhibit the desirable properties as quoted above from \cite[Section 1]{GXZ22}. For concreteness, this paper focuses on the increasing function $f(s) := \sqrt{2\lambda + 2\eta s^{m-1}}$, $s\ge 0$, for some fixed $\lambda,\eta>0$ and $m>1$. That is, \eqref{Langevin distri.} takes the form 
\begin{equation} \label{SDE mixed}
    dY_{t} = - \nabla \Psi(Y_{t}) \, dt + \sqrt{{ 2\lambda} + { 2\eta} \big(\rho^{Y_{t}}(Y_{t})\big)^{{m-1}}} \, dB_{t}, \quad \rho^{Y_{0}} = \rho_{0}. 
\end{equation}
This $f$ is chosen for a purpose: SDE \eqref{SDE mixed}, which is ``perturbed gradient descent'' for the non-convex $\Psi$ on $\R^d$, can be shown to be equivalent to ``gradient descent'' for the convex functional 
\begin{equation} \label{J}
    J_{\lambda, \eta}(\rho) :=
        \int_{\mathbb{R}^{d}} \Psi(y)\rho(y) + F(\rho(y))\, dy,\quad \rho\in\mathcal{D}_2(\mathbb{R}^{d}),
\end{equation}
where $F(s) := \lambda s\ln s+\frac{\eta}{m-1} s^m$ is convex on $[0,\infty)$ and $\Dc_2(\R^d)$ is the set of probability density functions with finite second moments; see Section~\ref{subsec:motivation} for details.  

The goal of this paper is to establish: (i) the existence of a unique solution $Y$ to SDE \eqref{SDE mixed}; (ii) the convergence of $Y_t$ to an invariant distribution as $t\to\infty$. The major challenge is that \eqref{SDE mixed} is not a standard McKean-Vlasov SDE, but one of the so-called {\it Nemytskii} type. 

In the vast literature of McKean-Vlasov SDEs (i.e., $dY_t = b(t,Y_t,\mu^{Y_t})\, dt + \sigma(t,Y_t,\mu^{Y_t})\, dB_t$, where $\mu^{Y_t}$ is the law of $Y_t$), the existence of a unique solution mostly requires (at least) continuity of $b,\sigma$ in the variables $y$ and $\mu$. However, when the dependence of $b,\sigma$ on $(y,\mu)$ is {\it mixed} through the density function $d\mu/d\Lc^d$ (where $\Lc^d$ is the Lebesgue measure on $\R^d$) and its value at $y$, i.e., 
\[
b(t,y,\mu) = \tilde b\Big(t,y,\frac{d\mu}{d\Lc^d}(y)\Big)\quad\hbox{and}\quad \sigma(t,y,\mu) = \tilde \sigma\Big(t,y,\frac{d\mu}{d\Lc^d}(y)\Big),
\]
there is in general no continuity of $b,\sigma$ in the $\mu$ variable; see \cite[p. 1904]{BR20}. Standard results and estimates, therefore, cannot be applied to this irregular class of SDEs, called McKean-Vlasov SDEs of the {Nemytskii} type. An early investigation \cite{JM98} obtains existence and uniqueness results when $\tilde b$ and $\tilde \sigma$ only depend on $({d\mu}/{d\Lc^d})(y)$. Fairly general $\tilde b$ has been treated lately, when $\tilde\sigma$ is independent of $({d\mu}/{d\Lc^d})(y)$; see e.g., \cite{HRZ21, IR23, Wang23, Le24}. As such, these studies do not readily cover \eqref{SDE mixed}. Aiming for a general theory that applies to {\it any} {Nemytskii}-type SDE, \cite{BR20} proposes that one first find a solution $u$ to the Fokker-Planck equation induced by the SDE: by rewriting the Fokker-Planck equation, usually highly nonlinear, as a Cauchy problem in $L^1(\R^d)$, the Crandall-Liggett theorem can be applied to exponentiate the Cauchy operator on $L^1(\R^d)$ via a careful discretization, which yields a solution to the Fokker-Planck equation; next, one can build from $u$ a solution $Y$ to the SDE by Trevisan's superposition principle \cite{Trevisan16}. Analyzing the $L^1(\R^d)$-discretization further shows that $Y_t$ converges to an invariant distribution as $t\to\infty$. See e.g., \cite{BR21, BR23, BR23-evolution, RR23} and the monograph \cite{BR-book-24} for this emerging new theory, and \cite{Huang23, Huang24} for its applications in machine learning.

Several boundedness assumptions, however, crucially underlie this emerging new theory. When applied to SDE \eqref{SDE mixed}, they turn into less desirable conditions, including $|\nabla\Psi| \in L^\infty(\R^d)$ and $(f^2)'\in L^\infty(\R^d)$. The former is quite restrictive for the non-convex $\Psi$ to be minimized and the latter simply fails under $f(s) = \sqrt{2\lambda + 2\eta s^{m-1}}$ in our case. 

In this paper, we instead approach the {Nemytskii}-type SDEs by the Wasserstein gradient flow theory, {\it without} relying on any boundedness condition. By the general existence of ``relaxed gradient flows'' in $\Pc_2(\R^d)$, the set of probability measures with finite second moments, for a broad class of the {Nemytskii}-type SDEs, the associated (nonlinear) Fokker-Planck equation has a solution; see \cite[Theorem 11.1.6 and Example 11.1.9]{Ambrosio08}. To apply this to SDE \eqref{SDE mixed}, no boundedness condition is required and we only need $F(s)$ in \eqref{J} to satisfy a doubling condition (Definition~\ref{def:doubling}). This is established in Lemma~\ref{lem:doubling}, which then gives a solution $\rho:[0,\infty)\to\Dc_2(\R^d)$ to the Fokker-Planck equation induced by \eqref{SDE mixed}; see Theorem~\ref{thm:sol. to FPE}. This solution $\rho$ is unique among a general class of curves in $\Dc_2(\R^d)$, as long as $\nabla\Psi$ is Lipschitz and has linear growth; see Theorem~\ref{thm:uniqueness FPE}. 

Note that this curve $\rho$ is constructed from a discretization called the ``minimizing movement'' scheme. Intriguingly, the aforementioned $L^1(\R^d)$-discretization fundamentally differs from this scheme, although they both yield a solution to the Fokker-Planck equation. The former directly discretizes the Fokker-Planck equation (as a differential equation in $L^1(\R^d)$); the latter by design reduces the value of $J_{\lambda,\eta}$ recursively in $\Dc_2(\R^d)$, and the Fokker-Planck equation only shows up as a by-product. It turns out that the minimizing movement scheme has several advantages.

First, a careful analysis of this scheme 
shows that $\{\rho(t)\}_{t\in[0,T]}$ is bounded in $L^m(\R^d)$ for any $T>0$; see Proposition~\ref{prop:L^m norm bdd}. This allows the use of Figalli-Trevisan's superposition principle in \cite{BRS21} (an upgrade of the one in \cite{Trevisan16}) to build a solution $Y$ to \eqref{SDE mixed} with $\rho^{Y_t} = \rho(t)$ for all $t\ge 0$; see Theorem~\ref{thm:sol. to Y}. In fact, such boundedness in $L^m(\R^d)$ characterizes the time-marginal laws of solutions to \eqref{SDE mixed}: for any solution $Y$ to \eqref{SDE mixed} such that $\{\rho^{Y_t}\}_{t\ge 0}$ fulfills this condition, $\rho^{Y_t}=\rho(t)$ for all $t\ge 0$; see Theorem~\ref{thm:uniqueness Y}. 
Second, as the minimizing movement scheme reduces the $J_{\lambda,\eta}$ value recursively, we quickly obtain ``$t\mapsto J_{\lambda,\eta}(\rho(t))$ is non-increasing'' in the limit; see Lemma~\ref{lem:J decreasing}. That is, we find that $J_{\lambda,\eta}$ is a Lyapunov function for the density flow $\{\rho(t)\}_{t\ge 0}$ almost {\it for free}, as opposed to the delicate work needed in \cite{BR23-evolution}. The monotonicity of $t\mapsto J_{\lambda,\eta}(\rho(t))$ directly implies that it is differentiable for a.e.\ $t>0$ and fulfills a variational inequality; 
see Proposition~\ref{prop:J'(t)}.

All the developments lead to our ultimate convergence result: the time-marginal law of a solution $Y$ to SDE \eqref{SDE mixed} (i.e., $\rho^{Y_t}\Lc^d = \rho(t)\Lc^d$) converges weakly to $\rho_\infty\Lc^d$ as $t\to\infty$, where $\rho_\infty$ is the unique minimizer of $J_{\lambda,\eta}$; see Theorem~\ref{thm:ultimate result}. 
In a nutshell, as $t \mapsto J_{\lambda,\eta}(\rho(t))$ is non-increasing and  bounded from below by $J_{\lambda,\eta}(\rho_\infty)$, we deduce that (a) $\{\rho(t)\Lc^d\}_{t\ge0}$ is tight and (b) $\frac{d}{dt} J_{\lambda,\eta}(\rho(t))\to 0$ as $t\to\infty$. Condition (b) facilitates a detailed analysis of $\{\rho(t)\Lc^d\}_{t\ge 0}$, relying on subdifferential calculus in $\Pc_2(\R^d)$ and compactness for functions of bounded variation, which reveals that its weak limit (as $t\to\infty$) must be $\rho_\infty\Lc^d$; see the proof of Theorem~\ref{thm:convergence to rho_infty} for details.

In a simple numerical example, we find that tuning $m>1$ in \eqref{SDE mixed} can {\it simultaneously} improve (i) the accuracy of the estimate of global minimizers and (ii) the rate of convergence to that estimate. This is in contrast to the Langevin dynamics \eqref{Langevin}, under which there is a tradeoff between (i) and (ii); 
see Section~\ref{sec:example} for details. 

Upon finishing this paper, we learned that \cite{BMV24} recently also approached minimizing $\Psi: M\to \R$ through a similar angle. There, the authors also consider a functional like \eqref{J}, but accommodate any convex $F$ satisfying a general condition and develop a gradient flow analysis for any such an $F$. The standing assumption in \cite{BMV24} is that the domain $M$ is compact and the main convergence result is established when $M$ is one-dimensional. By contrast, the general domain $\R^d$ is considered in our paper. Our convergence result holds without any compactness or dimensional constraint on the domain, although it is established under a specific $F$ specified below \eqref{J}. 

Considering the whole domain $\R^d$ also contributes to the literature of porous medium equations (PME). Note that the Fokker-Planck equation induced by \eqref{SDE mixed} can be viewed as a generalized PME with a drift. When the domain of this PME is bounded, classical studies \cite{DiB82, BH86} show the existence of a unique solution and its long-time convergence to equilibrium. Whether the same holds on the whole domain $\R^d$ is not fully understood. By extending the arguments in \cite{BH86} to $\R^d$, \cite{Carrillo01} constructs a solution on $\R^d$ (without looking into uniqueness) and proves its long-time convergence only when $\Psi$ is uniformly convex. Our study complements this by providing desired existence, uniqueness, and long-time convergence results on $\R^d$ without any convexity of $\Psi$. 

This paper is organized as follows. Section~\ref{sec:setup} motivates the SDE \eqref{SDE mixed}. Section~\ref{sec:subdifferential} obtains characterizations for subdifferentials of functions on $\Pc_2(\R^d)$. Section~\ref{sec:sol. to FPE} shows that the associated nonlinear Fokker-Planck equation has a unique solution, which is used to construct a solution $Y$ to \eqref{SDE mixed} in Section~\ref{sec:sol. to Y}. Section~\ref{sec:minimizer} derives the unique minimizer of $J_{\lambda,\eta}$ in \eqref{J}. Section~\ref{sec:convergence} proves the long-time convergence of $Y_t$ to the minimizer of $J_{\lambda,\eta}$. Section~\ref{sec:example} presents a numerical example. 


\section{The Setup}\label{sec:setup}
For any Polish space $S$, let $\mathcal B(S)$ be the Borel $\sigma$-algebra of $S$ and $\mathcal{P}(S)$ be the set of probability measures on $(S,\cB(S))$. Given Polish spaces $S_1$ and $S_2$, $\mu\in \mathcal P(S_1)$, and a Borel $\bm t:S_1\to S_2$, we define $\bm t_\# \mu\in\mathcal P(S_2)$, the pushforward of $\mu$ through $\bm t$, by $\bm t_\# \mu(B) := \mu(\bm t^{-1} (B))$ for all $B\in\mathcal B(S_2)$.
For $S=\R^d$ with $d\in\N$, we consider
$
\mathcal{P}_{2}(\mathbb{R}^{d}) := \big\{ \mu \in \mathcal{P}(\mathbb{R}^{d}) : {\int_{\R^d} |y|^2 d\mu(y) < \infty} \big\}
$
and equip it with the Wasserstein-2 metric, i.e.,  
${\cW_{2}(\mu, \nu)} := \inf_{\gamma \in \Gamma(\mu,\nu)} \big( \int_{\mathbb{R}^{d} \times \mathbb{R}^{d}} |x - y|^{2} \, d\gamma(x, y) \big)^{1/2}$ for $\mu,\nu\in \mathcal{P}_{2}(\mathbb{R}^{d})$, 
where $\Gamma(\mu,\nu)$ is the set of all $\gamma\in\mathcal P(\R^d\times\R^d)$ with marginals $\mu$ and $\nu$, which we call {\it transport plans} from $\mu$ to $\nu$. If a transport plan $\gamma^*\in \Gamma(\mu,\nu)$ attains the infimum in $\cW_{2}(\mu, \nu)$, we say it is {\it optimal}. 
We also consider the set $\cD(\R^d)$ of probability density functions on $\R^d$ and
$
\cD_2(\R^d) := \big\{ \rho \in \mathcal{D}(\mathbb{R}^{d}) : {\int_{\R^d} |y|^2 \rho(y)dy < \infty} \big\}. 
$
For any $\rho\in \mathcal D_2(\R^d)$, we write $\rho \mathcal L^d$ for the induced probability measure on $\R^d$ (which clearly lies in $\mathcal P_2(\R^d)$), where $\mathcal L^d$ stands for the Lebesgue measure on $\R^d$. 

\begin{remark}\label{rem:bm t}
Fix $\mu,\nu\in \Pc_2(\R^d)$. By \cite[Theorem 6.2.4]{Ambrosio08}, if $\mu\ll\Lc^d$, 
a unique optimal transport plan $\gamma^*\in\Gamma(\mu, \nu)$ exists 
and takes the form
$
\gamma^* = (\bm i\times \bm t_\mu^\nu)_\# \mu,
$
where $\bm i$ is the identity map on $\R^d$ and $\bm t_\mu^\nu:\R^d\to\R^d$ 
is Borel and satisfies ${\bm t_\mu^\nu}_\# \mu =\nu$  
(called the optimal transport map from $\mu$ to $\nu$). If $\nu\ll\Lc^d$ also holds, $\bm t_\mu^\nu$ is $\mu$-essentially injective and $\bm t_\nu^\mu \circ \bm t_\mu^\nu = \bm i$ $\mu$-a.e.; see \cite[Remark 6.2.11]{Ambrosio08}.  
\end{remark}


\subsection{Motivation and Objectives}\label{subsec:motivation}
In this paper, we consider minimizing a general $\Psi:\R^d\to\R$, assumed to be nonnegative and lie in $W^{1,1}_{\operatorname{loc}}(\R^d)$. As stated in Section~\ref{sec:intro}, the Langevin SDE \eqref{Langevin} converges as $t\to\infty$ 
to the distribution
\begin{equation}\label{rholambda}
\rho_\lambda(y) := Z_\lambda^{-1} {e^{-\frac1\lambda \Psi(y)}} 
\quad \ y\in\R^d,
\quad \hbox{with}\ \ Z_\lambda:=\int_{\R^d} e^{-\frac1\lambda \Psi(z)}dz,
\end{equation}
which centers around global minimizers of $\Psi$. That is, the Brownian perturbation $\sqrt{2\lambda}\, dB_t$ in \eqref{Langevin} is {\it well-balanced}---strong enough for $Y_t$ to escape from local minimizers and weak enough for $Y_t$ to admit a limiting distribution. It is natural to probe for the fundamental reason.

First, we note that $\rho_\lambda$ solves a convex optimization problem in $\Dc(\R^d)$. For any $\lambda>0$, 
\begin{equation}\label{G}
G_\lambda(\rho) := \int_{\R^d} \Psi(y)\rho(y)dy +\lambda\int_{\R^d}\rho(y) \ln\rho(y) dy, 
\qquad 
\rho\in\mathcal{D}(\R^d),
\end{equation}
is strictly convex on $\Dc(\R^d)$, as $\int_{\R^d} \Psi(y)\rho(y)dy$ and $\int_{\R^d}\rho(y) \ln\rho(y) dy$ are linear and strictly convex, respectively, in $\rho\in\Dc(\R^d)$. 
It is known that $\rho_\lambda$ in \eqref{rholambda} is the unique minimizer of $G_\lambda$; see e.g., \cite{JK96}. Namely, what minimizes the convex $G_\lambda:\mathcal{D}(\R^d)\to\R$ is the limiting distribution of \eqref{Langevin}. 
What underlies this coincidence can be made clear by Wasserstein subdifferential calculus.  

To see this, let us focus on $\Dc_2(\R^d)\subseteq \Dc(\R^d)$. By \cite[Lemma 10.4.1]{Ambrosio08}, under sufficient regularity, there is a unique subgradient of $G_\lambda$ at each $\rho\in\mathcal{D}_2(\R^d)$, which is a vector field given by $\nabla\Psi+\lambda {\nabla\rho}/{\rho}:\R^d\to\R^d$. Hence, the gradient-descent ODE for minimizing $G_\lambda$ can be formulated as
\begin{equation}\label{ODE'}
dY_t = -\left(\nabla \Psi(Y_t) + \lambda {\nabla\rho^{Y_t}(Y_t)}/{\rho^{Y_t}(Y_t)} \right)dt,\qquad \rho^{Y_0} =\rho_0\in\mathcal{D}_2(\R^d),
\end{equation}
where $\rho^{Y_t}$ denotes the density of $Y_t$. Here, $Y_0$ is sampled from a density $\rho_0$. Such randomness trickles through the deterministic dynamics in \eqref{ODE'}, making $Y_t$ a random variable at each $t> 0$. The {\it continuity equation} (or, {\it first-order Fokker-Planck equation}) 
for 
$u(t,y) := \rho^{Y_t}(y)$ 
is then
\begin{align}
u_t &= \text{div}\left(\left(\nabla \Psi + \lambda {\nabla u}/{u} \right) u \right)
=  \text{div}(u \nabla \Psi) + \lambda \Delta u. \label{FP}
\end{align}
Crucially, \eqref{FP} coincides with the {Fokker-Planck equation} for the Langevin SDE \eqref{Langevin}. 
It follows that 
``perturbed gradient descent \eqref{Langevin}'' for the non-convex $\Psi$ on $\R^d$ is equivalent to ``gradient descent \eqref{ODE'}'' for the convex $G_\lambda$ on $\mathcal{D}_2(\R^d)$, in the sense that they share the same Fokker-Planck equation (and thus both converge to the limiting distribution $\rho_\lambda$).  
This has intriguing implications: if we replace $\int_{\R^d}\rho(y) \ln\rho(y) dy$ in \eqref{G} by other convex functionals, $\sqrt{2\lambda}\, dB_t$ in \eqref{Langevin}, which stems from $\int_{\R^d}\rho(y) \ln\rho(y) dy$  through \eqref{ODE'}-\eqref{FP}, may change into new kinds of random perturbation. 

For instance, when $\int_{\R^d}\rho(y) \ln\rho(y) dy$ in \eqref{G} is replaced by $\int_{\R^d} \frac{\rho^m(y)}{m-1}dy$ with $m> 1$, by \cite[Lemma 10.4.1]{Ambrosio08} again, under sufficient regularity, there is a unique subgradient of $G_\lambda$ at each $\rho\in\mathcal{D}_2(\R^d)$, given by $\nabla\Psi+\lambda m \rho^{m-2} {\nabla\rho}:\R^d\to\R^d$. The resulting gradient-descent ODE is then
$dY_t = -\left(\nabla \Psi(Y_t) + \lambda m (\rho^{Y_t}(Y_t))^{m-2}  {\nabla\rho^{Y_t}(Y_t)} \right)dt$, $\rho^{Y_0} =\rho_0\in\mathcal{D}_2(\R^d)$. 
The continuity equation for $u(t,y) := \rho^{Y_t}(y)$ accordingly becomes
$
u_t =  \text{div}((\nabla \Psi + \lambda m u^{m-2}\nabla u)u )=  \text{div}(u \nabla \Psi) + \lambda \Delta u^m,
$
which is a {\it porous medium equation} (PME) with a drift. As it can be rewritten as $u_t = \text{div}(u \nabla \Psi) + \frac12\sum_{i=1}^d \frac{\partial^2}{\partial y_i^2} ((2\lambda u^{m-1}) u)$, 
the corresponding SDE is
\begin{equation}\label{PME SDE}
dY_t = -\nabla\Psi(Y_t)dt+\sqrt{2\lambda  (\rho^{Y_t}(Y_t))^{m-1}}\, dB_t,\quad \rho^{Y_0} = \rho_0\in\mathcal{D}_2(\R^d).
\end{equation}
As discussed below \eqref{Langevin distri.}, $\rho^{Y_t}(\cdot)$ here serves to give a sketch of the landscape imposed by $\Psi$ (i.e., the location and depth of local minimizers), so that the perturbation ``$\sqrt{2\lambda  (\rho^{Y_t}(Y_t))^{m-1}}\, dB_t$'' tends to be large near local minimizers and small otherwise. However, as PMEs are known to induce only finite propagation of mass (see \cite[Section 1]{Vazquez15}), it is likely that once ``$\sqrt{2\lambda  (\rho^{Y_t}(Y_t))^{m-1}}\, dB_t$'' forcefully pulls $Y_t =y$ away from a local minimizer, the perturbation wanes so quickly that $Y$ cannot move further away and is drawn back by ``$-\nabla\Psi(Y_t)dt$'' to the original local minimizer. 

In view of this, we propose to consider \eqref{SDE mixed}, a mixture of \eqref{Langevin} and \eqref{PME SDE}, where ``${2 \lambda}$'' provides {\it baseline} perturbation applied uniformly over $\R^d$ and ``$2\eta (\rho^{Y_t}(Y_t))^{m-1}$'' generates {\it location-wise} perturbation depending on the local landscape. By the same analysis in \eqref{G}-\eqref{FP}, ``perturbed gradient descent \eqref{SDE mixed}'' for the non-convex $\Psi$ on $\R^d$ is equivalent to ``gradient descent'' for the convex functional $J_{\lambda,\eta}$ on $\mathcal{D}_2(\R^d)$ in \eqref{J}, as they share the Fokker-Planck equation
    \begin{equation} \label{FPE mixed}
        \partial_{t} u - \Div\big(u \nabla \Psi  + \nabla \left(  \lambda u + \eta u^{m}  \right) \big) = 0,\quad u(0,\cdot) = \rho_0. 
    \end{equation}

This paper aims to: (i) show that 
\eqref{FPE mixed} has a unique solution $\rho$ (Section~\ref{sec:sol. to FPE}); (ii) use the superposition principle in \cite{Trevisan16, BRS21} to construct from $\rho$ a unique solution $Y$ to SDE \eqref{SDE mixed} (Section~\ref{sec:sol. to Y}); (iii) establish the long-time convergence of $Y_t$ to the minimizer of $J_{\lambda,\eta}$ (Sections~\ref{sec:minimizer} and \ref{sec:convergence}). To achieve these goals, we need to first characterize the subdifferentials of functions on $\Pc_2(\R^d)$. 
 

\section{Preliminaries: Subdifferentials in $\Pc_2(\R^d)$}\label{sec:subdifferential}
The next definition is in line with \cite[Definition 10.1.1]{Ambrosio08}. Recall $\bm t_\mu^\nu$ from Remark~\ref{rem:bm t}.

\begin{definition}\label{def:subdifferential}
Consider a lower semicontinuous $\phi:\mathcal P_2(\R^d)\to (-\infty,\infty]$ such that 
\begin{equation}\label{D(phi)}
D(\phi):=\{\mu\in\Pc_2(\R^d):\phi(\mu)<\infty\}\neq\emptyset\quad\hbox{and}\quad D(\phi)\subseteq\{\mu\in\Pc_2(\R^d):\mu\ll \Lc^d\}. 
\end{equation}
Given $\mu\in D(\phi)$, we say $\xi:\R^d\to\R^d$ is a {subgradient} of $\phi$ at $\mu$ if
$\xi\in L^2(\R^d,\mu)$ and
\begin{align}\label{subdifferential'}
    \phi(\nu)  \geq \phi(\mu) + \int_{\mathbb{R}^{d}} \xi(x)\cdot (\bm t_\mu^\nu(x) - x) \, d\mu(x)+ o\big(\cW_{2}(\mu, \nu)\big),\quad\forall\nu \in \mathcal{P}_{2}(\mathbb{R}^{d}). 
\end{align}
The set of all such {$\xi: \R^d\to\R^d$}, denoted by $\partial \phi(\mu)$, is called the {subdifferential} of $\phi$ at $\mu\in\mathcal P_2(\R^d)$.
\end{definition}

\begin{remark}\label{rem:subdifferential convex}
In Definition~\ref{def:subdifferential}, if $\phi$ is additionally convex along geodesics in $\mathcal P_2(\R^d)$ (see \cite[Definition 9.1.1]{Ambrosio08}), 
then 
 the arguments in \cite[Section 10.1.1, Part B]{Ambrosio08} indicate that $\xi\in\partial\phi(\mu)$ if and only if 
   $ G(\nu)  \geq G(\mu) + \int_{\mathbb{R}^{d}} \xi(x)\cdot (\bm t_\mu^\nu(x) - x) \, d\mu(x)$ for all $\nu \in \mathcal{P}_{2}(\mathbb{R}^{d}).$
\end{remark}

\begin{remark}\label{rem:in Tan}
For the case $\partial \phi(\mu) = \{\xi\}$ (i.e., $\partial \phi(\mu)$ is a singleton), $\xi$ must lie in the tangent space $\operatorname{Tan}_\mu \Pc_2(\R^d)$, i.e., the closure of $\{\nabla\varphi: \varphi\in C^\infty_c(\R^d)\}$ in $L^2(\R^d,\mu)$. Indeed, if $\xi\notin \operatorname{Tan}_\mu \Pc_2(\R^d)$, then $\xi = \zeta + w$ for some $\zeta\in \operatorname{Tan}_\mu \Pc_2(\R^d)$ and nonzero $w\in \operatorname{Tan}^\bot_\mu \Pc_2(\R^d):= \{u\in L^2(\R^d,\mu):\nabla\cdot(u\mu)=0\}$. When we plug $\xi = \zeta + w$ into \eqref{subdifferential'}, as $w\in \operatorname{Tan}^\bot_\mu \Pc_2(\R^d)$, $\int_{\R^d} w(x)\cdot (\bm t_\mu^\nu(x) - x) \, d\mu(x)$ is zero due to \cite[(8.4.3) and Theorem 8.5.1]{Ambrosio08}. It then implies $\zeta\in\partial \phi(\mu)$, contradicting $\partial \phi(\mu) = \{\xi\}$.
\end{remark}


To characterize the subdifferentials of two specific types of functions on $\Pc_2(\R^d)$ (i.e., \eqref{Ic} and \eqref{Kc} below), we need the ``doubling condition'' below (cf.\ \cite[(10.4.23)]{Ambrosio08}). 

\begin{definition}\label{def:doubling}
We say $G:[0,\infty) \rightarrow \mathbb{R}$ fulfills the {doubling} condition if for some $C > 0$, 
\begin{equation}\label{doubling}
    G(x+y) \leq C(1 + G(x) + G(y))\quad \hbox{for all $x,y \ge 0$}.
\end{equation}
\end{definition}

\begin{lemma}\label{lem:in subdifferential}
Let $G:[0,\infty)\to\R$ be convex and differentiable, satisfy the doubling condition \eqref{doubling}, and have superlinear growth at infinity. Assume additionally that
\begin{equation}\label{G/s^alpha}
G(0) = 0,\quad \liminf_{s\downarrow 0} \frac{G(s)}{s^\alpha} >-\infty\ \ \hbox{for some}\ \alpha>\frac{d}{d+2};
\end{equation}
\begin{equation}\label{makes I_G convex}
s\mapsto s^d G(s^{-d})\ \ \hbox{is convex and nondecreasing on $(0,\infty)$} . 
\end{equation}
Consider 
\begin{equation} \label{Ic}
   \mathcal I_G(\mu) :=
    \begin{cases}
       \int_{\R^d}G(\rho(y))dy,&\quad \hbox{if}\ \mu = \rho\mathcal L^d, \\
        \infty, &\quad \hbox{if otherwise},
	\end{cases}
	\quad\forall \mu\in\Pc_2(\R^d), 
\end{equation}
and
\begin{equation}\label{L_G}
L_G(s) := s G'(s) - G(s),\quad s\ge 0.
\end{equation}
For $V:\R^d\to\R$ lower semicontinuous, bounded from below, and locally Lipschitz, consider
\begin{equation} \label{Kc}
   \mathcal K(\mu) :=
    \begin{cases}
       \int_{\R^d} V(y)\rho(y) + G(\rho(y))dy,&\quad \hbox{if}\ \mu = \rho\mathcal L^d, \\
        \infty, &\quad \hbox{if otherwise},
	\end{cases}
	\quad\forall \mu\in\Pc_2(\R^d).
\end{equation}
Then, for any $\rho\in\mathcal D_2(\R^d)$, the following two relations hold:
\begin{align}
&w\in\partial \mathcal K(\rho\mathcal L^d)\ \ \ \ \implies\ \
L_G(\rho)\in W^{1,1}_{\operatorname{loc}}(\R^d)\ \ \hbox{and}\ \  w=\nabla V+{\nabla L_G(\rho)}/{\rho}\in L^2(\R^d,\rho \mathcal L^d).\label{subdiff. K} \\
&w\in\partial \mathcal I_G(\rho\mathcal L^d)\ \ \iff\ \
L_G(\rho)\in W^{1,1}_{\operatorname{loc}}(\R^d)\ \ \hbox{and}\ \
w={\nabla L_G(\rho)}/{\rho}\in L^2(\R^d,\rho \mathcal L^d).\label{subdiff. I_G}
\end{align}
\end{lemma}

\begin{proof}
Note that \eqref{subdiff. K} is obtained in \cite[Example 11.1.9]{Ambrosio08}; see \cite[(11.1.36)]{Ambrosio08} and the arguments below it. 
For \eqref{subdiff. I_G}, the sufficiency follows from \eqref{subdiff. K} by taking $V\equiv 0$, while the necessity is a consequence of the converse part of \cite[Theorem 10.4.6]{Ambrosio08} and \cite[Remark 9.3.8]{Ambrosio08}. 
\end{proof}


Before applying Lemma~\ref{lem:subdiff.=} to $J_{\lambda,\eta}$ in \eqref{J}, we need to first examine more closely what the integral in \eqref{J} means. 
For any $\lambda,\eta>0$ and $m>1$, let us write $F:[0,\infty)\to\R$ in \eqref{J} as  
\begin{equation} \label{F}
    F(s) = H(s) + P(s)\quad\forall s\ge 0, 
\end{equation}
\begin{equation}\label{H and P}
\hbox{with}\qquad H(s):=\lambda s \ln s\ \ \forall s>0,\ \ H(0):= H(0+)=0;\qquad P(s):= \frac{\eta}{m-1} s^{m}\ \ \forall s\ge 0.
\end{equation}
With $\Psi, P\ge 0$, we interpret $J_{\lambda,\eta}$ in \eqref{J} precisely as 
\[
J_{\lambda,\eta}(\rho) = \int_{\R^d}\Psi \rho dy+\int_{\R^d} H^+(\rho) dy - \int_{\R^d} H^-(\rho) dy +\int_{\R^d} P(\rho) dy. 
\]

\begin{remark}\label{rem:J}
For any  $\alpha\in (0,1)$, 
$\lim_{s\downarrow 0} {H(s)}/{s^\alpha} = \lim_{s\downarrow 0}\frac{-\lambda}{1-\alpha} s^{1-\alpha} = 0$. 
Hence, $H$ fulfills \eqref{G/s^alpha}. This implies that for any $\rho\in\mathcal D_2(\R^d)$,  $H^-(\rho)$ is integrable; see \cite[Remark 9.3.7]{Ambrosio08}. This, along with $\Psi,P\ge 0$, 
entails $J_{\lambda,\eta}(\rho) > -\infty$.  
Moreover, $J_{\lambda, \eta}(\rho)<\infty$ iff $\Psi\rho, (\rho\ln(\rho))^+, \rho^m\in L^1(\R^d)$. 
\end{remark}

We further define $\mathcal J_{\lambda,\eta}:\Pc_2(\R^d)\to \R\cup\{\infty\}$ by 
\begin{equation} \label{Jc}
   \mathcal J_{\lambda, \eta}(\mu) :=
    \begin{cases}
       J_{\lambda,\eta}(\rho),&\quad \hbox{if}\ \mu = \rho\mathcal L^d, \\
        \infty, &\quad \hbox{if otherwise},
	\end{cases}
	\quad\forall \mu\in\Pc_2(\R^d). 
\end{equation}
Observe from \eqref{Ic} and \eqref{F} that 
\begin{equation}\label{J=G+I_F}
\mathcal J_{\lambda,\eta}(\mu) = \mathcal G(\mu) + \mathcal I_F(\mu)\qquad\forall \mu\in\Pc_2(\R^d),
\end{equation}
where $\mathcal G:\Pc_2(\R^d)\to \R\cup\{\infty\}$ is defined by
   $\mathcal G(\mu) :=  \int_{\R^d}\Psi(y) d\mu(y)$ for $\mu\in\Pc_2(\R^d).$

\begin{remark}\label{rem:LSC of Jc}
If $\Psi\ge 0$ is lower semicontinuous, $\mathcal{G}$ is lower semicontinuous w.r.t.\ weak convergence in $\Pc(\R^d)$. As $\Ic_F$ is lower semicontinuous w.r.t.\ the same topology (\cite[Theorem 15.8]{Ambrosio24}), 
so is $\Jc_{\lambda, \eta}$. 
\end{remark}

Careful estimates show that $F$ in \eqref{F} satisfies the doubling condition \eqref{doubling} provided that $\lambda>0$ is not too large, as stated in the next result (whose proof is relegated to Section~\ref{sec:proof of lem:doubling}). 

\begin{lemma}\label{lem:doubling}
For any $\lambda\in (0,1/\ln 2]$, $H$ in \eqref{H and P} fulfills \eqref{doubling} with $C= \frac{1}{1-\ln 2}$. For any $\lambda\in (0,e/2]$, $\eta>0$, and $m > 1$, $F$ in \eqref{F} fulfills \eqref{doubling} with $C= \frac{1}{1-\ln 2}\vee (2^m-1)$.
\end{lemma}

With $F$ satisfying the doubling condition, Lemma~\ref{lem:in subdifferential} can be applied to characterize $\partial\Ic_F(\mu)$. 

\begin{lemma}\label{lem:subdiff.=}
Fix $\lambda, \eta>0$, and $m>1$. For any $\rho\in\Dc(\R^d)$, the following hold:
(i) if $L_F(\rho)\in W^{1,1}_{\operatorname{loc}}(\R^d)$, then $\rho, \rho^m\in W^{1,1}_{\operatorname{loc}}(\R^d)$ with $\nabla \rho^m = m\rho^{m-1}\nabla\rho$; (ii) if additionally $\lambda<e/2$, $\rho\in\Dc_2(\R^d)$, and $\nabla L_F(\rho)/\rho \in L^2(\R^d,\rho\mathcal L^d)$, then 
\begin{align}
\partial\Ic_F(\rho\Lc^d) &= \{\nabla L_F(\rho)/\rho\}= \{\lambda\nabla\rho/\rho +\eta\nabla \rho^m/\rho\},\label{subdiff. I_F=}\\
\partial\Ic_P(\rho\Lc^d) &= \{\nabla L_P(\rho)/\rho\} = \{\eta\nabla \rho^m/\rho\}.\label{subdiff. I_P=}
\end{align}
\end{lemma}

\begin{proof}
(i) By \eqref{L_G} and \eqref{F}, $L_F(s)=\lambda s+\eta s^m$ is strictly increasing with $L_F'(s)=\lambda +\eta m s^{m-1}\ge \lambda>0$ for $s\ge 0$. Thus, $L_F^{-1}$ is well-defined and lies in $C^1([0,\infty))$ with $(L_F^{-1})'$ bounded (By the inverse function theorem, $(L_F^{-1})'(r) = 1/(L_F'(L_F^{-1}(r)))\in (0, 1/\lambda]$). As we can write $\rho= L_F^{-1}(L_F(\rho))$, where $L_F^{-1}\in C^1([0,\infty))$ has bounded first derivative and $L_F(\rho)\in W^{1,1}_{\operatorname{loc}}(\R^d)$, the chain rule for weak derivatives is applicable and it asserts the existence of the weak derivative $\nabla\rho\in L^1_{\operatorname{loc}}(\R^d)$, given by $\nabla\rho = (L_F^{-1})'(L_F(\rho)) \nabla L_F(\rho) = ({1}/{L_F'(\rho)}) \nabla L_F(\rho)$. 
This yields $\rho\in W^{1,1}_{\operatorname{loc}}(\R^d)$ and 
\begin{equation}\label{nabla L_F}
\nabla L_F(\rho) = L_F'(\rho) \nabla\rho = (\lambda + \eta m \rho^{m-1})\nabla \rho.
\end{equation}
With $L_F(\rho), \rho\in  W^{1,1}_{\operatorname{loc}}(\R^d)$, we right away get $\rho^m = \frac1\eta (L_F(\rho) - \lambda\rho)\in W^{1,1}_{\operatorname{loc}}(\R^d)$ with $\nabla \rho^m = \frac1\eta (\nabla L_F(\rho)-\lambda \nabla\rho) = m\rho^{m-1}\nabla\rho$, where the last equality follows from \eqref{nabla L_F}.  

(ii) Note that $F$ in \eqref{F} satisfies all the conditions in Lemma~\ref{lem:in subdifferential}, thanks in part to Lemma~\ref{lem:doubling}. Hence, with $L_F(\rho)\in W^{1,1}_{\operatorname{loc}}(\R^d)$ and $\nabla L_F(\rho)/\rho \in L^2(\R^d,\rho\mathcal L^d)$, \eqref{subdiff. I_G} directly implies $\partial\Ic_F(\rho\Lc^d) = \{\nabla L_F(\rho)/\rho\}= \{\lambda\nabla\rho/\rho +\eta\nabla \rho^m/\rho\}$, where the last equality follows from \eqref{nabla L_F} and part (i). Note that $P$ in \eqref{H and P} also satisfies all the conditions in Lemma~\ref{lem:in subdifferential}; particularly, the doubling condition \eqref{doubling} is fulfilled trivially due to \eqref{(x+y)^m}. By \eqref{L_G} and part (i), $L_P(\rho) = \eta \rho^m\in W^{1,1}_{\operatorname{loc}}(\R^d)$.  
Moreover, 
\begin{align}
\int_{\R^d}\frac{|\nabla L_P(\rho)|^2}{\rho^2} \rho dy = \int_{\R^d} \frac{\eta m \rho^{m-1}|\nabla\rho|^2}{\rho^2} \rho dy \le \int_{\R^d} \frac{|\nabla L_F(\rho)|^2}{\rho^2} \rho dy<\infty, 
\end{align}
where the equality follows from part (i), the first inequality is due to \eqref{nabla L_F}, and the finiteness results from $L_F(\rho)/\rho\in L^2(\R^d,\rho\mathcal L^d)$. Thus, we conclude $\nabla L_P(\rho)/\rho\in L^2(\R^d, \rho\mathcal L^d)$. With $L_P(\rho)\in W^{1,1}_{\operatorname{loc}}(\R^d)$ and $\nabla L_P(\rho)/\rho \in L^2(\R^d,\rho\mathcal L^d)$, \eqref{subdiff. I_G} implies $\partial\Ic_P(\rho\Lc^d) = \{\nabla L_P(\rho)/\rho\}= \{\eta\nabla \rho^m/\rho\}$.
\end{proof}

When $\nabla\Psi$ has linear growth, we can further characterize $\partial \Jc_{\lambda,\eta}(\mu)$ based on 
Lemma~\ref{lem:subdiff.=}. 

\begin{corollary}\label{coro:subdiff.=}
Fix $\lambda\in(0,e/2)$, $\eta>0$, and $m>1$. Assume $\Psi\in C^1(\R^d)$ and 
\begin{equation}\label{Psi' linear growth}
\exists K>0\ \ \hbox{s.t.}\ \ |\nabla\Psi(y)|\le K(1+|y|)\quad \forall y\in\R^d.
\end{equation}
Then, for any $\rho\in\mathcal D_2(\R^d)$, 
\begin{align}
w\in\partial \mathcal J_{\lambda,\eta}(\rho\mathcal L^d)\ \ \iff\ \
L_F(\rho)\in W^{1,1}_{\operatorname{loc}}(\R^d)\ \ \hbox{and}\ \ 
w=\nabla \Psi+{\nabla L_F(\rho)}/{\rho}\in L^2(\R^d,\rho \mathcal L^d). 
\label{subdiff. J} 
\end{align}
Hence, if $\partial \mathcal J_{\lambda,\eta}(\rho\mathcal L^d)\neq\emptyset$, we have $\partial \mathcal J_{\lambda,\eta}(\rho\mathcal L^d) = \{\nabla \Psi+{\nabla L_F(\rho)}/{\rho}\}= \{\nabla\Psi+\lambda\nabla\rho/\rho +\eta\nabla \rho^m/\rho\}$ as well as \eqref{subdiff. I_F=} and \eqref{subdiff. I_P=}. 
\end{corollary}

\begin{proof}
As $\rho\in\mathcal D_2(\R^d)$, \eqref{Psi' linear growth} implies $\int_{\R^d}|\nabla\Psi(y)|^2 \rho(y) dy\le C_K\int_{\R^d} (1+|y|^2)\rho(y) <\infty$ for some $C_K>0$ depending on $K>0$. Hence, $\nabla\Psi \in L^2(\R^d,\rho\mathcal L^d)$.  

The sufficiency in \eqref{subdiff. J} follows directly from Lemma~\ref{lem:in subdifferential}; see \eqref{subdiff. K}, with $(V,G)$ therein taken to be $(\Psi,F)$. We will thus focus on proving the necessity. 
Assume $L_F(\rho)\in W^{1,1}_{\operatorname{loc}}(\R^d)$ and $w=\nabla \Psi+{\nabla L_F(\rho)}/{\rho}\in L^2(\R^d,\rho \mathcal L^d)$. As $w, \nabla\Psi \in L^2(\R^d,\rho\mathcal L^d)$, we have ${\nabla L_F(\rho)}/{\rho} = w-\nabla\Psi\in L^2(\R^d,\rho \mathcal L^d)$. Hence, Lemme~\ref{lem:subdiff.=} (ii) directly gives $ \partial\mathcal I_F(\rho\mathcal L^d) = \{\nabla L_F(\rho)/\rho\}$. 
On the other hand, 
by Example 1 in \cite[Section 5.2.2]{CD-book-18-I}, under \eqref{Psi' linear growth}, the ``L-derivative'' of $\mathcal G$ in \eqref{J=G+I_F} at any $\mu\in\Pc_2(\R^d)$ (see \cite[Definition 5.22]{CD-book-18-I}) is $\nabla \Psi$. As $\nabla \Psi$ is continuous, $\mathcal G$ is continuously L-differentiable. Then, \cite[Theorem 5.64]{CD-book-18-I} asserts that the L-derivative of $\mathcal G$ is the unique element of $\partial\mathcal G(\mu)$, i.e., $\partial\mathcal G(\mu)=\{\nabla\Psi\}$ $\forall \mu\in\Pc_2(\R^d)$. With $\partial\mathcal G(\rho\mathcal L^d)=\{\nabla\Psi\}$ and $\partial\mathcal I_F(\rho\mathcal L^d) = \{\nabla L_F(\rho)/\rho\}$, 
we deduce from $\mathcal J_{\lambda,\eta} = \mathcal G + \mathcal I_F$ (recall \eqref{J=G+I_F}) and Definition~\ref{def:subdifferential} that $\nabla\Psi + \nabla L_F(\rho)/\rho\in \partial \Jc_{\lambda,\eta}(\rho\Lc^d)$, as desired. 

Finally, if $\partial \mathcal J_{\lambda,\eta}(\rho\mathcal L^d)\neq\emptyset$, \eqref{subdiff. J} directly implies  $L_F(\rho)\in W^{1,1}_{\operatorname{loc}}(\R^d)$, $w=\nabla \Psi+{\nabla L_F(\rho)}/{\rho}\in L^2(\R^d,\rho \mathcal L^d)$, and $\partial \mathcal J_{\lambda,\eta}(\rho\mathcal L^d) = \{\nabla \Psi+{\nabla L_F(\rho)}/{\rho}\}= \{\nabla\Psi+\lambda\nabla\rho/\rho +\eta\nabla \rho^m/\rho\}$, where the last equality follows from \eqref{nabla L_F}. With $L_F(\rho)\in W^{1,1}_{\operatorname{loc}}(\R^d)$ and ${\nabla L_F(\rho)}/{\rho} = w-\nabla\Psi\in L^2(\R^d,\rho \mathcal L^d)$, Lemme~\ref{lem:subdiff.=} (ii) directly gives \eqref{subdiff. I_F=} and \eqref{subdiff. I_P=}. 
\end{proof}


\section{Solving the Nonlinear Fokker-Planck Equation \eqref{FPE mixed}}\label{sec:sol. to FPE}
We first specify precisely what constitutes a solution to the Fokker-Planck Equation \eqref{FPE mixed}. 
    
\begin{definition}\label{def:sol. to FPE}
Given $\rho_0\in\mathcal D(\R^d)$, we say $\rho:[0,\infty)\to\Dc(\R^d)$ is a solution to \eqref{FPE mixed} 
if the following hold: (i) $t\mapsto\rho(t)$ is continuous on $[0,\infty)$ under the weak-$*$ topology in $L^1(\R^d)$, with $\rho(0)=\rho_0$; 
(ii) $L_F(\rho(t))\in W^{1,1}_{\operatorname{loc}}(\R^d)$ for all $t> 0$; (iii) for any $\varphi \in C_{c}^{\infty}((0,\infty) \times \mathbb{R}^{d})$, 
    \begin{equation}\label{distri. sense}
        \int_{0}^{\infty} \int_{\mathbb{R}^{d}}   \left\{\rho\partial_{t} \varphi - \Big(\rho \nabla \Psi  + \nabla \big(  \lambda \rho + \eta \rho^{m}  \big)  \Big) \cdot \nabla \varphi \right\} \, dy \, dt = 0.
    \end{equation}
\end{definition}

\begin{remark}
By \eqref{L_G} and \eqref{F}, $L_F(s) = \lambda s +\eta s^m$ for $s\ge 0$. Hence, for any $t>0$, the condition ``$L_F(\rho(t))\in W^{1,1}_{\operatorname{loc}}(\R^d)$'' in Definition~\ref{def:sol. to FPE} is to ensure that $\int_{\R^d}\big(\nabla (\lambda \rho + \eta \rho^{m} )\cdot \nabla \varphi\big)(t,y)\, dy$ in \eqref{distri. sense} is well-defined. By Lemma~\ref{lem:subdiff.=} (i), it also implies  $\rho(t), \rho^m(t)\in W^{1,1}_{\operatorname{loc}}(\R^d)$.
\end{remark}


To state our existence result for \eqref{FPE mixed}, we still need absolutely continuous curves in $\Pc_2(\R^d)$. 

\begin{definition}
Let $I\subseteq [0,\infty)$ be a time interval. For any $p\ge 1$, a curve {$\bm\nu : I \rightarrow \mathcal{P}_{2}(\mathbb{R}^{d})$} is (locally) absolutely continuous with order $p$,  
if there exists {$m\in L^p(I)$} ($m\in L^p_{\operatorname{loc}}(I)$) such that  
       $\cW_{2}(\bm\nu_{t},\bm\nu_{s}) \leq \int_{s}^{t} {m(z)} \, dz$ for $s,t \in I$ with $s \leq t.$
We denote by {$AC^p(I;\mathcal P_2(\R^d))$} ($AC^p_{\operatorname{loc}}(I;\mathcal P_2(\R^d))$) the set of all such curves. For $p=1$, we will simply write {$AC(I;\mathcal P_2(\R^d))$} ($AC_{\operatorname{loc}}(I;\mathcal P_2(\R^d))$).
\end{definition}


\begin{theorem}\label{thm:sol. to FPE}
Fix $\lambda\in(0,e/2)$, $\eta>0$, and $m>1$. Assume $\Psi\in C^1(\R^d)$. 
For any $\rho_0\in \mathcal D_2(\R^d)$ such that $J_{\lambda,\eta}(\rho_0) <\infty$, there exists a solution $\rho:[0,\infty)\to\Dc_2(\R^d)$ to the Fokker-Planck equation \eqref{FPE mixed} 
such that $\{\mu_t\}_{t\ge 0}$, with $\mu_t := \rho(t) \mathcal L^d\in\Pc_2(\R^d)$, 
lies in $AC^2_{\operatorname{loc}}([0,\infty);\Pc_2(\R^d))$ and 
\begin{equation}\label{v}
v_\rho(t,y) := - \left(\nabla\Psi +\lambda \nabla\rho/\rho + \eta \nabla\rho^m/\rho\right)(t,y)\quad \forall (t, y)\in[0,\infty)\times\R^d
\end{equation}
fulfills the estimate
\begin{equation}\label{int v^2}
\int_0^T \int_{\R^d} \left|v_\rho(t,y)\right|^2 \rho(t,y)\, dy dt <\infty,\quad\forall T>0.
\end{equation}  
If we further assume \eqref{Psi' linear growth}, then for $\Lc^1$-a.e.\ $t\ge 0$,
\begin{align*}
&\partial\mathcal J_{\lambda,\eta}(\rho(t)\mathcal L^d) = \{-v_\rho(t,\cdot)\} =\{(\nabla\Psi +\lambda \nabla\rho/\rho + \eta\nabla \rho^m/\rho)(t,\cdot)\},\\
\partial\Ic_F(\rho(t)&\Lc^d) = \{(\lambda\nabla\rho/\rho +\eta\nabla \rho^m/\rho)(t,\cdot)\},\qquad
\partial\Ic_P(\rho(t)\Lc^d) = \{(\eta\nabla \rho^m/\rho)(t,\cdot)\}.
\end{align*}
\end{theorem}

\begin{proof}
Note that $F$ in \eqref{F} is convex, differentiable, has superlinear growth at infinity, and fulfills \eqref{G/s^alpha} (as $H$ fulfills \eqref{G/s^alpha}; see Remark~\ref{rem:J}) and the doubling condition \eqref{doubling} (by Lemma~\ref{lem:doubling}). This, along with $\Psi\ge 0$ and $\Psi\in C^1(\R^d)$, implies that all conditions in \cite[Example 11.1.9]{Ambrosio08} are fulfilled. The arguments therein assert that \cite[Theorem 11.1.6]{Ambrosio08} can be applied, which yield $\{\mu_t\}_{t\ge 0}\in AC^2_{\operatorname{loc}}([0,\infty);\Pc_2(\R^d))$ and $\rho:[0,\infty)\to\Dc_2(\R^d)$ such that: (i) $\mu_t = \rho(t) \mathcal L^d$ and $L_F(\rho(t))\in W^{1,1}_{\operatorname{loc}}(\R^d)$ for all $t\ge 0$ (recall \eqref{L_G} for $L_F$), (ii) $\rho$ is a solution (in the sense of distributions) to
       $ \partial_{t} \rho + \Div\big(\rho \bar v \big) = 0$ with $\rho(0) = \rho_0,$
where $\bar v(t,y):= -(\nabla\Psi +\nabla L_F(\rho)/\rho)(t,y)$ satisfies $\int_0^T \int_{\R^d}|\bar v(t,y)|^2\rho(t,y) \, dy dt <\infty$ for all $T>0$; 
see \cite[(11.1.36), (11.1.37), (11.1.38)]{Ambrosio08}. As $\{\mu_t\}_{t\ge 0}\in AC^2_{\operatorname{loc}}([0,\infty);\Pc_2(\R^d))$, $t\mapsto\mu_t$ is continuous on $[0,\infty)$ in $\Pc_2(\R^d)$, whence  weakly continuous on $[0,\infty)$ in $\Pc(\R^d)$.  
This is equivalent to the weak-$*$ continuity of $t\mapsto \rho(t)$ on $[0,\infty)$ in $L^1(\R^d)$. 
By \eqref{L_G} and \eqref{F}, $L_F(s) =\lambda s+\eta s^m$. This, along with \eqref{nabla L_F} (applicable as  $L_F(\rho(t))\in W^{1,1}_{\operatorname{loc}}(\R^d)$), shows that $\bar v = -(\nabla\Psi +\nabla L_F(\rho)/\rho)$ coincides with $v_\rho$ in \eqref{v}. With $\bar v =v_\rho$, we obtain the estimate \eqref{int v^2} and the equation $\partial_{t} \rho + \Div\big(\rho \bar v \big) = 0$ becomes \eqref{FPE mixed}. We thus conclude that $\rho$ is a solution to \eqref{FPE mixed} (cf.\ Definition~\ref{def:sol. to FPE}). 
Finally, note that \eqref{int v^2} implies that $v_\rho(t,\cdot) = \bar v(t,\cdot) =-(\nabla\Psi +\nabla L_F(\rho(t))/\rho(t))$ lies in $L^2(\R^d,\rho(t)\Lc^d)$ for $\Lc^1$-a.e.\ $t\ge 0$. Thus, under \eqref{Psi' linear growth}, Corollary~\ref{coro:subdiff.=} implies the remaining assertions. 
\end{proof}


Next, we will establish a stability result for solutions to the Fokker-Planck equation \eqref{FPE mixed} (i.e., Proposition~\ref{prop:uniqueness FPE} below), from which uniqueness of solutions and a flow property can be derived. The next technical lemma, taken from \cite[Lemma 7.1]{Hwang24}, will prove useful. 

\begin{lemma} \label{lem:W_2'}
Fix any $T>0$. For $i = 1,2$, let $\{\mu_{t}^{i}\}_{t\in(0,T)} \in AC((0,T);\mathcal{P}_{2}(\mathbb{R}^{d}))$ satisfy
    \begin{equation}\label{conti. eqn.}
        \partial_{t} \mu_{t}^{i} + \nabla \cdot (v^{i} \mu_{t}^{i}) = 0\quad \hbox{for}\ \ t\in (0,T)\qquad \hbox{in the sense of distributions}, 
    \end{equation}
    for some $v^{i}:(0,T)\times\R^d\to \R^d$ such that 
    $
    \int_0^T \int_{\R^d} |v^i(t,y)|^2 d\mu_t^i(y) dt < \infty. 
    $
     Then, for $\mathcal{L}^1$-a.e.\ $t \in (0,T)$, 
    \begin{equation}
        \frac{d}{dt}\Wc_{2}^{2}(\mu_{t}^{1}, \mu_{t}^{2})\leq
         2 \int_{\mathbb{R}^{d} \times \mathbb{R}^{d}} (v_{1}(t,x) - v_{2}(t,y))\cdot (x- y) \, d\gamma_{t}(x,y), \quad \forall \gamma_{t} \in \Gamma_{0}(\mu_{t}^{1}, \mu_{t}^{2}). 
    \end{equation}
\end{lemma}

Our stability result for \eqref{FPE mixed} will be stated for the following class of functions. 

\begin{definition}\label{def:C}
For any $\lambda, \eta>0$ and $m>1$, let $\mathcal C$ be the set of $\theta:[0,\infty)\to \Dc_2(\R^d)$ such that
\begin{equation}\label{C condition}
\hbox{$L_F(\theta(t))\in W^{1,1}_{\operatorname{loc}}(\R^d)$ for all $t\ge 0$}\quad \hbox{and}\quad  
\hbox{$v_\theta(t,y)$ given by \eqref{v}  satisfies \eqref{int v^2}.}
\end{equation}
\end{definition}

\begin{proposition}\label{prop:uniqueness FPE}
Fix $\lambda, \eta>0$ and $m>1$. Assume $\Psi\in C^1(\R^d)$, \eqref{Psi' linear growth}, and that $\nabla \Psi:\R^d\to\R^d$ is $K$-Lipschitz for some $K>0$. For $i=1,2$, take $\rho_0^i\in \Dc_2(\R^d)$ such that $J_{\lambda,\eta}(\rho_0^i)<\infty$ and let $\rho^i\in\mathcal C$ be a solution to the Fokker-Planck equation \eqref{FPE mixed} with $\rho_0 = \rho_0^i$ (Definition~\ref{def:sol. to FPE}).
Then, the flow of measures $\mu_{t}^{i} := \rho^{i}(t) \mathcal{L}^{d}\in \Pc_2(\R^d)$, $t\ge 0$, satisfy
\begin{equation}\label{W_2<=W_2}
\Wc_{2}(\mu_{t}^{1}, \mu_{t}^{2}) \leq e^{2K(t-s)} \Wc_{2}(\mu_{s}^{1}, \mu_{s}^{2}),\qquad \forall 0< s\le t. 
\end{equation}
If additionally $\mu^i_t\to\mu^i_0$ in $\Pc_2(\R^d)$ as $t\downarrow 0$, then \eqref{W_2<=W_2} holds also for $s=0$; hence, if $\rho^1_0 = \rho^2_0$ $\Lc^d$-a.e., then $\mu_{t}^{1}=\mu_{t}^{2}$ (or equivalently, $\rho^1(t) = \rho^2(t)$ $\Lc^d$-a.e.) for all $t\ge 0$. 
\end{proposition}

\begin{proof}
Fix $T>0$. For $i=1,2$, when restricted to $t\in (0,T)$, \eqref{FPE mixed} (with $\rho=\rho^i$) can be equivalently written as \eqref{conti. eqn.}, where $v^i = v_{\rho^i}$ is given by \eqref{v}. As $v^i$ satisfies \eqref{int v^2} (due to $\rho^i\in\mathcal C$), we have $v^i(t,\cdot)\in L^2(\R^d,\mu^i_t)$ for $\mathcal L^1$-a.e.\ $t\in (0,T)$ and 
 $\|v^i(t,\cdot)\|_{L^2(\R^d,\mu_t)}\in L^2((0,T))$. This particularly implies $\|v^i(t,\cdot)\|_{L^2(\R^d,\mu_t)}\in L^1((0,T))$. As $t\mapsto \mu^i_t$ is weakly continuous on $(0,T)$ (by Definition~\ref{def:sol. to FPE} (i)) and $\rho^i$ fulfills \eqref{conti. eqn.} with $\|v^i(t,\cdot)\|_{L^2(\R^d,\mu_t)}\in L^1((0,T))$, the converse part of \cite[Theorem 8.3.1]{Ambrosio08} asserts $\mu^i\in AC((0,T);\Pc_2(\R^d))$. Hence, 
by Lemma~\ref{lem:W_2'} and $T>0$ being arbitrary, for $\Lc^1$-a.e.\ $t>0$, 
        \begin{align}
            \frac{d}{dt}\Wc^2_2(\mu^1_t,\mu^2_t) &\leq 2  \int_{\R^{d} \times \R^{d}} \bigg(\Big(\lambda \frac{\nabla \rho^{2}}{\rho^2}+\eta \frac{\nabla (\rho^2)^m}{\rho^2}\Big) (t,y)-\Big(\lambda \frac{\nabla \rho^{1}}{\rho^1} +\eta \frac{\nabla (\rho^1)^m}{\rho^1}\Big)(t,x)\bigg)\notag\\
&\hspace{3.7in}  \cdot(x- y) \, d\gamma_{t}(x,y)\notag \\
            &\hspace{1.7in}+2 \int_{\mathbb{R}^{d} \times \mathbb{R}^{d}}(\nabla\Psi(y) - \nabla \Psi(x))\cdot (x-y) \, d\gamma_{t}(x,y),\label{W_2' our case}
        \end{align}
 where $\gamma_t\in \Gamma_0(\mu^1_t,\mu^2_t)$ is now unique and given by $\gamma_t = \big(\bm i\times\bm t_{\mu^1_t}^{\mu^2_t}\big)_\# \mu^1_t$; 
 recall Remark~\ref{rem:bm t}. 
 
Under \eqref{Psi' linear growth}, as argued in the proof of Corollary~\ref{coro:subdiff.=}, $\nabla\Psi\in L^2(\R^d, \mu^i_t)$ for all $t\ge 0$. Also, for any $t\ge0$, 
as $L_F(\rho^i(t)) \in W^{1,1}_{\operatorname{loc}}(\R^d)$ (thanks to $\rho^i\in \mathcal C$), \eqref{nabla L_F} is applicable and it gives $\nabla L_F(\rho^i(t))/\rho^i(t) = \lambda \nabla\rho^i(t)/\rho^i(t) +\eta \nabla(\rho^i(t))^m/\rho^i(t) = -v^i(t,\cdot) - \nabla\Psi\in L^2(\R^d,\mu^i_t)$. Since $v^i(t,\cdot)$ (resp.\ $\nabla \Psi$) lies in $L^2(\R^d,\mu^i_t)$ for $\mathcal L^1$-a.e.\ (resp.\ all) $t\ge 0$, we get $\nabla L_F(\rho^i_t)/\rho^i_t\in L^2(\R^d,\mu^i_t)$ for $\mathcal L^1$-a.e.\ $t\ge 0$. Hence, by Lemma~\ref{lem:subdiff.=} (ii), $\partial\Ic_F(\mu^i_t) = \{\lambda\nabla\rho^i(t)/\rho^i(t)+\eta\nabla (\rho^i(t))^m/\rho^i(t)\}$ for $\Lc^1$-a.e.\ $t\ge 0$. As $F$ in \eqref{F} fulfills \eqref{makes I_G convex}, $\Ic_F$ is convex along geodesics in $\Pc_2(\R^d)$ (by \cite[Proposition 9.3.9]{Ambrosio08}). Thus, by Remark~\ref{rem:subdifferential convex}, for $\Lc^1$-a.e.\ $t\ge 0$,
    \begin{align*} 
        \mathcal{I}_F(\mu^2_t) &\geq \mathcal{I}_F(\mu^1_t) + \int_{\R^d} \bigg(\Big(\lambda \frac{\nabla \rho^{1}}{\rho^1} +\eta \frac{\nabla (\rho^1)^m}{\rho^1}\Big)(t,x)\bigg) \cdot \Big(\bm t_{\mu^1_t}^{\mu^2_t}(x)- x\Big)\, d\mu^1_t(x);\\
        \mathcal{I}_F(\mu^1_t) &\geq \mathcal{I}_F(\mu^2_t) + \int_{\R^d} \bigg(\Big(\lambda \frac{\nabla \rho^{2}}{\rho^2} +\eta \frac{\nabla (\rho^2)^m}{\rho^2}\Big)(t,x)\bigg) \cdot \Big(\bm t_{\mu^2_t}^{\mu^1_t}(x)- x\Big)\, d\mu^2_t(x)\\
        &= \mathcal{I}_F(\mu^2_t) + \int_{\R^d} \bigg(\Big(\lambda \frac{\nabla \rho^{2}}{\rho^2} +\eta \frac{\nabla (\rho^2)^m}{\rho^2}\Big)\Big(t,\bm t_{\mu^1_t}^{\mu^2_t}(x)\Big)\bigg) \cdot \Big(x- \bm t_{\mu^1_t}^{\mu^2_t}(x)\Big)\, d\mu^1_t(x),
        \end{align*}
where the equality follows from $\mu^2_t = \big(\bm t_{\mu^1_t}^{\mu^2_t}\big)_\# \mu^1_t$, the change of variable formula for pushforward measures, and $\bm t_{\mu^2_t}^{\mu^1_t}\circ\bm t_{\mu^1_t}^{\mu^2_t}=\bm i$ $\mu^1_t$-a.e.\ (recall Remark~\ref{rem:bm t}). Summing up the two inequalities yields
      $      0 \geq  \int_{\R^{d} \times \R^{d}} \Big(\big(\lambda \frac{\nabla \rho^{2}}{\rho^2} +\eta \frac{\nabla (\rho^2)^m}{\rho^2}\big)(t,y) - \big(\lambda \frac{\nabla \rho^{1}}{\rho^1}+\eta \frac{\nabla (\rho^1)^m}{\rho^1}\big) (t,x)\Big) \cdot(x- y) \, d\gamma_{t}(x,y)$, 
   where we used $\gamma_t = \big(\bm i\times\bm t_{\mu^1_t}^{\mu^2_t}\big)_\# \mu^1_t$. Plugging this into \eqref{W_2' our case} leads to
        \begin{align*}
            \frac{d}{dt}\Wc^2_2(\mu^1_t,\mu^2_t) &\leq 2 \int_{\mathbb{R}^{d} \times \mathbb{R}^{d}}(\nabla\Psi(y) - \nabla \Psi(x))\cdot (x-y) \, d\gamma_{t}(x,y)\\
            &\le 2 \int_{\mathbb{R}^{d} \times \mathbb{R}^{d}} K|x-y|^2 \, d\gamma_{t}(x,y) = 2K \Wc_2^2(\mu^1_t,\mu^2_t),\quad \hbox{for $\Lc^1$-a.e.\ $t>0$}, 
        \end{align*}
where the second inequality follows from $\nabla\Psi$ being $K$-Lipschitz. As $t\mapsto \Wc^2_2(\mu^1_t,\mu^2_t)$ is continuous on $[s,\infty)$ for all $s>0$ (thanks to $\mu^i\in AC((0,T);\Pc_2(\R^d))$ for all $T>0$),  Gr\"{o}nwall's inequality is applicable, so that the previous inequality implies \eqref{W_2<=W_2}. If additionally $\mu^i_t\to\mu^i_0$ in $\Pc_2(\R^d)$ as $t\downarrow 0$, then $t\mapsto \Wc^2_2(\mu^1_t,\mu^2_t)$ is continuous on $[0,\infty)$, such that \eqref{W_2<=W_2} holds even for $s=0$. 
\end{proof}

The first consequence of Proposition~\ref{prop:uniqueness FPE} is the uniqueness of solutions to \eqref{FPE mixed}. 

\begin{theorem}\label{thm:uniqueness FPE}
Fix $\lambda\in (0,e/2)$, $\eta>0$ and $m>1$. Suppose that $\Psi\in C^1(\R^d)$, \eqref{Psi' linear growth} holds, and $\nabla \Psi:\R^d\to\R^d$ is Lipschitz. For any $\rho_0\in \Dc_2(\R^d)$ such that $J_{\lambda,\eta}(\rho_0)<\infty$, $\rho:[0,\infty)\to\Dc_2(\R^d)$ in Theorem~\ref{thm:sol. to FPE} is the unique solution to the Fokker-Planck equation \eqref{FPE mixed} among 
\begin{equation}\label{C_{rho_0}}
\mathcal C_{\rho_0} := \{ \theta\in \mathcal C: \theta(t)\Lc^d\to\rho_0\Lc^d\ \hbox{in}\ \Pc_2(\R^d)\ \hbox{as}\ t\downarrow 0\}. 
\end{equation}
Specifically, if $\tilde\rho\in\mathcal C_{\rho_0}$ is a solution to \eqref{FPE mixed}, then $\tilde\rho(t)=\rho(t)$ $\Lc^d$-a.e.\ for all $t\ge 0$.
\end{theorem}

\begin{proof}
By Theorem~\ref{thm:sol. to FPE}, $\rho$ is a solution to \eqref{FPE mixed} such that $L_F(\rho(t)) \in W^{1,1}_{\operatorname{loc}}(\R^d)$ for all $t\ge 0$ (as part of Definition~\ref{def:sol. to FPE}), \eqref{int v^2} holds,  and $\mu_t := \rho(t) \mathcal L^d$ lies in $AC^2_{\operatorname{loc}}([0,\infty);\Pc_2(\R^d))$. The last condition indicates that $t\mapsto\mu_t$ is continuous on $[0,\infty)$ in $\Pc_2(\R^d)$, which in turn implies the weak-$*$ continuity of $t\mapsto\rho(t)$ on $[0,\infty)$ in $L^1(\R^d)$ and $\mu_t\to\mu_0$ in $\Pc_2(\R^d)$ as $t\downarrow 0$. We thus conclude $\rho\in\mathcal C_{\rho_0}$. The desired uniqueness then follows from Proposition~\ref{prop:uniqueness FPE}.
\end{proof}

The second consequence of Proposition~\ref{prop:uniqueness FPE} is the flow property of the (unique) solution to \eqref{FPE mixed}. To state this clearly, for any $\rho_0\in\Dc_2(\R^d)$ such that $J_{\lambda,\eta}(\rho_0)<\infty$, we introduce the notation
\[
S(t)[\rho_0] := \rho(t)\in\Dc_2(\R^d)\quad \forall t\ge 0,
\]
where $\rho:[0,\infty)\to\Dc_2(\R^d)$ is the solution to \eqref{FPE mixed} obtained in Theorem~\ref{thm:sol. to FPE}. 

\begin{corollary}\label{coro:flow property}
Assume the conditions in Proposition~\ref{prop:uniqueness FPE}. Then, for any $\rho_0\in \Dc_2(\R^d)$ such that $J_{\lambda,\eta}(\rho_0)<\infty$, $S(t+s)[\rho_0] = S(t)[S(s)[\rho_0]]$ $\Lc^d$-a.e.\ for all $s,t\ge 0$  
\end{corollary}  

\begin{proof}
Fix $s\ge 0$. As $t\mapsto S(t)[\rho_0]$ is a solution to \eqref{FPE mixed} satisfying the properties in Theorem~\ref{thm:sol. to FPE}, we deduce from Definition~\ref{def:sol. to FPE} that $t\mapsto S(s+t)[\rho_0]$ is also a solution to \eqref{FPE mixed} (with initial density $S(s)[\rho_0]$) satisfying the properties in Theorem~\ref{thm:sol. to FPE}. Now that $t\mapsto S(s+t)[\rho_0]$ and $t\mapsto S(t)[S(s)[\rho_0]]$ are both solutions to \eqref{FPE mixed} satisfying the properties in Theorem~\ref{thm:sol. to FPE} and they share the same initial density $S(s)[\rho_0]$, Proposition~\ref{prop:uniqueness FPE} implies $S(t+s)[\rho_0] = S(t)[S(s)[\rho_0]]$ $\Lc^d$-a.e.\ for all $t\ge 0$. 
\end{proof}


\section{Solving the Density-Dependent Langevin SDE \eqref{SDE mixed}}\label{sec:sol. to Y}
In this section, we will build a solution $Y$ to SDE \eqref{SDE mixed} using the solution $\rho:[0,\infty)\to \Dc_2(\R^d)$ to the Fokker-Planck equation \eqref{FPE mixed} from Theorem~\ref{thm:sol. to FPE}. To this end, we need to first establish suitable $L^p(\R^d)$ estimates for $\rho(t)$, $t\ge 0$. Recall that $\rho(t)$ is constructed using \cite[Theorem 11.1.6]{Ambrosio08}, which relies on the ``minimizing movement'' scheme in \cite[Section 11.1.3]{Ambrosio08}, to be briefly reviewed below. 

Fix any time step $\tau>0$. For any $\mu,\nu\in\Pc_2(\R^d)$, consider the function 
\begin{equation}\label{Theta}
\Theta(\tau,\mu;\nu) := \frac{1}{2\tau} \Wc_2^2(\nu,\mu)+\Jc_{\lambda,\eta}(\nu).
\end{equation}
Given $\mu_0\in \Pc_2(\R^d)$ such that $\Jc_{\lambda,\eta}(\mu_0)<\infty$, set $M_\tau^0 := \mu_0$ and define $\{M_\tau^n\}_{n\in\N}$ recursively by 
\begin{equation}\label{M in argmin}
M_\tau^n \in \argmin_{} \{ \Theta(\tau,M_\tau^{n-1};\nu): \nu\in\Pc_2(\R^d)\}\quad \forall n\in\N. 
\end{equation}
In view of $\eqref{Jc}$ and $\eqref{J}$, $M_\tau^n$ must admit a density function, i.e., 
\begin{equation}\label{M has density}
\hbox{for all $n\ge 0$},\quad  M_\tau^n  = \rho^n_\tau \Lc^d\quad \hbox{for some}\ \ \rho^n_\tau\in\Dc_2(\R^d).
\end{equation}
Piecewise-constant curves $\overline M_\tau:(0,\infty)\to \Pc_2(\R^d)$ and $\overline \rho_\tau:(0,\infty)\to\Dc_2(\R^d)$ can then be defined by
\begin{equation}\label{overline M}
\overline M_\tau(t) := M_\tau^n\quad \hbox{and}\quad \overline \rho_\tau(t) := \rho_\tau^n\qquad \hbox{for}\ t\in((n-1)\tau,\tau],\ n\in\N.
\end{equation}
For each $n\in\N$, let $\bm t_\tau^n$ be the optimal transport map from $M^n_\tau$ to $M^{n-1}_\tau$ (recall Remark~\ref{rem:bm t}) and define the discrete velocity vector $V_\tau^n := ({\bm i - \bm t_\tau^n})/{\tau}$. By \cite[Lemma 10.1.2, Theorem 10.4.2]{Ambrosio08}, 
\begin{equation}\label{-V in subdiff. J}
-V_\tau^n = {(\bm t_\tau^n-\bm i)}/{\tau} \in \partial \Jc_{\lambda,\eta}(M^n_\tau). 
\end{equation}

\begin{remark}\label{rem:Thm 11.1.6}
As $\Jc_{\lambda,\eta}$ satisfies all the conditions in \cite[Theorem 11.1.6]{Ambrosio08} (thanks to $F$ fulfilling all the conditions in \cite[Example 11.1.9]{Ambrosio08}), the theorem states that there exist $\{\tau_k\}_{k\in\N}$ with $\tau_k\downarrow 0$ and $\mu:(0,\infty)\to\Pc_2(\R^d)$ such that: (a) $\overline M_{\tau_k}(t)\to \mu_t$ weakly in $\Pc(\R^d)$ for all $t\ge0$; (b) $\mu_t = \rho(t)\Lc^d$, $t>0$, where $\rho:[0,\infty)\to \Dc_2(\R^d)$ is the solution to 
\eqref{FPE mixed} in Theorem~\ref{thm:sol. to FPE}.
Now, for each fixed $t>0$, $\overline M_{\tau_k}(t) = \overline\rho_{\tau_k}(t,\cdot)\Lc^d$ by \eqref{M has density} and \eqref{overline M}. Hence, (a) and (b) above assert $\overline\rho_{\tau_k}(t)\Lc^d\to\rho(t)\Lc^d$ weakly in $\Pc(\R^d)$, which is equivalent to 
$\int_{\R^d} \phi(y) \overline\rho_{\tau_k}(t,y) dy \to \int_{\R^d} \phi(y) \rho(t,y) dy$ for all $\phi\in C^\infty_c(\R^d).$ 
\end{remark}

Looking into the minimizing movement scheme closely yields an $L^m(\R^d)$ bound for $\rho(t)$. 

\begin{proposition}\label{prop:L^m norm bdd}
Fix $\lambda\in(0,e/2)$, $\eta>0$, and $m>1$. Assume $\Psi\in C^2(\R^d)$ with $(\Delta \Psi)^+ \in L^{\infty}(\mathbb{R}^{d})$, \eqref{Psi' linear growth}, and 
\begin{equation}\label{Psi increasing}
\exists R>0\ \ \hbox{s.t.}\ \ \Psi(k y) \ge \Psi(y)\quad \forall k\ge 1\ \hbox{and $|y|\ge R$}. 
\end{equation}
For any $\rho_0\in\Dc_2(\R^d)$ such that $J_{\lambda,\eta}(\rho_0)<\infty$, let $\rho:[0,\infty)\to\Dc_2(\R^d)$ be the solution to \eqref{FPE mixed} obtained in Theorem~\ref{thm:sol. to FPE}. Then, $\rho(t) \in L^{m}(\mathbb{R}^{d})$ for all $t\ge 0$. Moreover,  
   \begin{equation}
        ||\rho(t)||_{L^{m}(\mathbb{R}^{d})} \leq e^{\frac{m-1}{m} A t} ||\rho_{0}||_{L^{m}(\mathbb{R}^{d})}\quad \forall t\ge 0,\quad \hbox{with}\  \ A:= \|(\Delta\Psi)^+\|_{L^\infty(\R^d)}. 
    \end{equation}
\end{proposition}

\begin{proof} 
By Remark~\ref{rem:J}, $J_{\lambda,\eta}(\rho_0)<\infty$ readily implies $\rho(0)=\rho_0\in L^m(\R^d)$. 
For any $\tau>0$, recall $\{M^n_\tau\}_{n\in\N\cup\{0\}}$ in $\Pc_2(\R^d)$ from \eqref{M in argmin}, where we take $\mu_0 = \rho_0\mathcal L^d$. For any $n\in\N$,
\begin{align}
J_{\lambda,\eta}(\rho^n_\tau) &= \Jc_{\lambda,\eta}(M^n_\tau) \le \frac{1}{2\tau} \Wc_2^2(M^n_\tau,M^{n-1}_\tau)+\Jc_{\lambda,\eta}(M^n_\tau)\notag\\
& = \Theta(\tau, M^{n-1}_\tau; M^n_\tau)\le \Theta(\tau, M^{n-1}_\tau; M^{n-1}_\tau) = \Jc_{\lambda,\eta}(M^{n-1}_\tau) = J_{\lambda,\eta}(\rho^{n-1}_\tau),\label{J^n<=J^n-1}
\end{align}
where the first and last equalities follow from \eqref{M has density} and \eqref{Jc}, the second inequality is due to \eqref{M in argmin}, and the second equality holds by the definition of $\Theta$ in \eqref{Theta}. This implies $J_{\lambda,\eta}(\rho^n_\tau)\le J_{\lambda,\eta}(\rho^0_\tau) = J_{\lambda,\eta}(\rho_0)   <\infty$ and thus $\rho^n_\tau\in L^m(\R^d)$ (by Remark~\ref{rem:J}) for all $n\in\N$. Then, by \eqref{Ic} and \eqref{H and P}, \begin{equation}\label{I_P<infty}
\Ic_P(M_\tau^n) = \frac{\eta}{m-1}\int_{\R^d} (\rho^n_\tau)^m dy<\infty,\quad \forall n\in\N\cup\{0\}.  
\end{equation}

Now, for each $n\in\N$, thanks to \eqref{-V in subdiff. J} and \eqref{M has density}, Corollary~\ref{coro:subdiff.=} implies 
\begin{equation}\label{dudu}
-V^{n}_\tau = \nabla \Psi+\lambda\nabla\ln\rho_\tau^n +\eta\nabla (\rho_\tau^n)^m/\rho_\tau^n\qquad \hbox{and}\qquad \partial\Ic_P(M_\tau^n) = \{\eta\nabla (\rho_\tau^n)^m/\rho_\tau^n\}.  
\end{equation}
As $P$ in \eqref{H and P} fulfills \eqref{makes I_G convex}, $\Ic_P$ is convex along geodesics in $\Pc_2(\R^d)$ (by \cite[Proposition 9.3.9]{Ambrosio08}). Hence, by Remark~\ref{rem:subdifferential convex} and $\partial\Ic_P(M_\tau^n) = \{\eta\nabla (\rho_\tau^n)^m/\rho_\tau^n\}$,
    \begin{align} \label{I_P>=I_P+}
        \mathcal{I}_P(M_{\tau}^{n-1}) &\geq \mathcal{I}_P(M_{\tau}^{n}) +\eta \int_{\mathbb{R}^{d}} \nabla (\rho_\tau^n)^m(y) \cdot (\bm{t}_{\tau}^{n}(y) - y)\, dy\notag\\
        & =  \mathcal{I}_P(M_{\tau}^{n}) + \eta\tau \int_{\mathbb{R}^{d}} \nabla (\rho_\tau^n)^m(y) \cdot \big(\nabla \Psi+\lambda\nabla\ln\rho_\tau^n +\eta\nabla (\rho_\tau^n)^m/\rho_\tau^n\big)(y)\, dy\notag\\
        &\ge \mathcal{I}_P(M_{\tau}^{n}) + \eta\tau \int_{\mathbb{R}^{d}} (\nabla (\rho_\tau^n)^m\cdot \nabla \Psi)(y)\, dy,
    \end{align}
where the second line stems from the definition $V^n_\tau = ({\bm i - \bm t_\tau^n})/{\tau}$ 
and the first part of \eqref{dudu}, and the last inequality holds as 
$
\nabla (\rho_\tau^n)^m(y) \cdot \big(\lambda\nabla\ln\rho_\tau^n +\eta\nabla (\rho_\tau^n)^m/\rho_\tau^n\big) = \lambda m (\rho_\tau^n)^{m-2}|\nabla\rho_\tau^n|^2 +\eta|\nabla(\rho^n_\tau)^m|^2\ge 0. 
$
Note that the integral on the right-hand side of \eqref{I_P>=I_P+} is well-defined. Indeed, under \eqref{Psi' linear growth}, we have $\nabla\Psi \in L^2(\R^d,\rho^n_\tau\Lc^d)$, as argued in the proof of Corollary~\ref{coro:subdiff.=}. Also, by Definition~\ref{def:subdifferential}, $\eta\nabla (\rho_\tau^n)^m/\rho_\tau^n\in  \partial\Ic_P(M_\tau^n)$ readily implies $\nabla (\rho_\tau^n)^m/\rho_\tau^n\in  L^2(\R^d,\rho^n_\tau\Lc^d)$. Thus, by H\"{o}lder's inequality,
\begin{align}\label{bala}
\int_{\mathbb{R}^{d}} \big|\nabla (\rho_\tau^n)^m\cdot \nabla \Psi\big|\, dy &= \int_{\mathbb{R}^{d}} \Big|({\nabla (\rho_\tau^n)^m}/{\rho_\tau^n})\cdot \nabla \Psi\Big|\ \rho_\tau^n\, dy\notag\\ 
&\le \| {\nabla (\rho_\tau^n)^m}/{\rho_\tau^n}\|_{L^2(\R^d,\rho_\tau^n\Lc^d)} \|\nabla\Psi\|_{L^2(\R^d,\rho_\tau^n\Lc^d)}<\infty.
\end{align}
It follows that
        \begin{align}
            \int_{\mathbb{R}^{d}} \nabla (\rho_{\tau}^{n})^{m} &\cdot \nabla \Psi \, dy = \lim_{R \rightarrow \infty} \int_{B_{R}(0)} \nabla (\rho_{\tau}^{n})^{m} \cdot \nabla \Psi \, dy \notag\\
            &= \lim_{R \rightarrow \infty} \int_{\partial B_{R}(0)} (\rho_{\tau}^{n})^{m} ( \nabla \Psi \cdot \mathbf{n}) \, dy - \int_{B_{R}(0)} (\rho_{\tau}^{n})^{m} \Delta \Psi \, dy \notag\\
            &\geq \lim_{R \rightarrow \infty} - \int_{B_{R}(0)} (\rho_{\tau}^{n})^{m} \Delta \Psi \, dy \ge - \int_{\mathbb{R}^{d}} (\rho_{\tau}^{n})^{m} \Delta \Psi \, dy \geq - A \int_{\mathbb{R}^{d}} ( \rho_{\tau}^{n})^{m} \, dy,
            \label{esti}
        \end{align}
        with $A:=\|(\Delta\Psi)^+\|_{L^\infty(\R^d)}$. Here, the first equality follows from the dominated convergence theorem (applicable due to \eqref{bala}), the second equality is due to integration by parts in $\R^d$ (with $\mathbf{n}$ therein the outward unit normal vector on $\partial_R B(0)$), the first inequality holds as \eqref{Psi increasing} implies $\nabla \Psi \cdot \mathbf{n} \ge0$, and the second inequality follows from the reverse Fatou lemma (applicable due to $\rho^n_\tau\in L^m(\R^d)$ and $(\Delta\Psi)^+\in L^\infty(\R^d)$). By plugging \eqref{I_P<infty} into \eqref{I_P>=I_P+} and using the estimate \eqref{esti}, we get 
\begin{align*}
\int_{\R^d} (\rho^{n-1}_\tau)^m\, dy \ge \int_{\R^d} (\rho^n_\tau)^m\, dy -\tau(m-1) A \int_{\mathbb{R}^{d}} ( \rho_{\tau}^{n})^{m} \, dy = \big(1-\tau(m-1)A\big) \int_{\R^d} (\rho^n_\tau)^m\, dy. 
\end{align*}
Hence, for any $\tau< 1/((m-1)A)$, we have $\int_{\R^d} (\rho^n_\tau)^m\, dy \le \big(1-\tau(m-1)A\big)^{-1} \int_{\R^d} (\rho^{n-1}_\tau)^m\, dy$ for all $n\in\N$. Applying this estimate recursively yields
$\int_{\R^d} (\rho^n_\tau)^m\, dy \le \big(1-\tau(m-1)A\big)^{-n} \int_{\R^d} \rho_{0}^m\, dy$ for all $n\in\N$. 
By recalling $\overline\rho_\tau(t)$ in \eqref{overline M}, this implies
$\int_{\R^d} (\overline \rho_\tau(t))^m\, dy \le \big(1-\tau(m-1)A\big)^{-\lceil t/\tau \rceil} \int_{\R^d} \rho_{0}^m\, dy$ for all $t>0,$
where $\lceil \cdot \rceil$ denotes the ceiling function, which in turn gives
\begin{equation}\label{bar rho estimate}
\|\overline \rho_\tau(t)\|_{L^m(\R^d)} \le \big(1-\tau(m-1)A\big)^{-\lceil t/\tau \rceil/m} \|\rho_{0}\|_{L^m(\R^d)},\qquad\forall t>0,
\end{equation}
Take $\{\tau_k\}_{k\in\N}$ with $\tau_k\downarrow 0$ from Remark~\ref{rem:Thm 11.1.6}. 
For any fixed $t>0$, \eqref{bar rho estimate} implies that $\{\overline \rho_{\tau_k}(t)\}_{k\in\N}$ is bounded in $L^m(\R^d)$. Hence, there is a subsequence (without relabeling) that converges weakly in $L^m(\R^d)$, i.e., there exists $\rho^*\in L^m(\R^d)$ such that
$\int_{\R^d} \phi(y) \overline\rho_{\tau_k}(t,y) dy \to \int_{\R^d} \phi(y) \rho^*(y) dy$ for all $\phi\in L^{\frac{m}{m-1}}(\R^d)$. 
This, along with the convergence at the end of Remark~\ref{rem:Thm 11.1.6}, implies 
$\rho^* = \rho(t)$ $\Lc^d$-a.e., whence $ \overline\rho_{\tau_k}(t)\to \rho(t)$ weakly in $L^m(\R^d)$. As $||\cdot||_{L^{m}(\mathbb{R}^{d})}$ is lower semicontinuous under the weak topology of $L^{m}(\mathbb{R}^{d})$,
       $ ||\rho(t)||_{L^{m}(\mathbb{R}^{d})}\le \liminf_{k\to\infty}\|\overline \rho_{\tau_k}(t))\|_{L^m(\R^d)} \le e^{\frac{m-1}{m}At} ||\rho_{0}||_{L^{m}(\mathbb{R}^{d})},$
   where the second inequality follows from taking $\tau=\tau_k\downarrow 0$ in \eqref{bar rho estimate}. 
\end{proof}

To state our existence result for SDE \eqref{SDE mixed}, let us first specify what constitutes a solution to it. 

\begin{definition}
Let $(\Omega,\cF, \{\cF_t\}_{t\ge 0} ,\P)$ be a filtered probability space that supports a $d$-dimensional Brownian motion $B$. An adapted process $Y:[0,\infty)\times \Omega\to\R^d$ is said to be a solution to \eqref{SDE mixed}, if the density function $\theta(t)\in\mathcal D(\R^d)$ of $Y_t$ exists for all $t\ge 0$, with $\theta(0)=\rho_0$ $\mathcal L^d$-a.e., and $Y$ satisfies 
\begin{equation} \label{SDE mixed ell}
    dY_{t} = - \nabla \Psi(Y_{t}) \, dt + \sqrt{{ 2\lambda} + { 2\eta} \big(\theta(t,Y_{t})\big)^{{m-1}}} \, dB_{t}\quad  \forall t> 0, \qquad \hbox{a.s.}
\end{equation}
\end{definition}

\begin{theorem}\label{thm:sol. to Y}
Fix $\lambda\in(0,e/2)$, $\eta>0$, and $m>1$. Assume $\Psi\in C^2(\R^d)$ with $(\Delta \Psi)^+ \in L^{\infty}(\mathbb{R}^{d})$, \eqref{Psi' linear growth}, and \eqref{Psi increasing}. 
Then, for any $\rho_{0} \in \Dc_2(\mathbb{R}^{d})$ such that $J_{\lambda,\eta}(\rho_0)<\infty$, there exists a solution $Y$ to \eqref{SDE mixed} with $\rho^{Y_t} = \rho(t)$ for all $t\ge 0$, where $\rho:[0,\infty)\to\Dc_2(\R^d)$ is the solution to the Fokker-Planck equation \eqref{FPE mixed} 
obtained in Theorem~\ref{thm:sol. to FPE}. 
\end{theorem}

\begin{proof}
Recall $\rho: [0,\infty)\to\Dc_2(\R^d)$ from Theorem~\ref{thm:sol. to FPE} and consider the auxiliary SDE
\begin{equation} \label{SDE mixed rho}
    dY_{t} = - \nabla \Psi(Y_{t}) \, dt + \sqrt{{2\lambda} + {2\eta} \rho^{m-1}(t, Y_{t})} \, dB_{t}, \quad \rho^{Y_{0}} = \rho_{0} \in \mathcal{D}_2(\mathbb{R}^{d}),
\end{equation}
whose associated Fokker-Planck equation is
    \begin{equation} \label{FPE mixed u^*}
        \partial_{t} u - \Div\big(u \nabla \Psi  + \nabla \left(  \lambda u + \eta \rho^{m-1}u  \right) \big) = 0,\quad u(0,\cdot) = \rho_0. 
    \end{equation}
As $\rho$ is a solution to \eqref{FPE mixed}, it is readily a solution to \eqref{FPE mixed u^*}. Hence, to apply Figalli-Trevisan's superposition principle \cite[Theorem 1.1]{BRS21}, it remains to check \cite[(1.3)]{BRS21}, which in our case is 
\begin{equation}\label{to show for SPP}
\int_0^T\int_{\R^d} \frac{{ 2\lambda} + { 2\eta} \rho^{m-1}(t, y) + |y \cdot\nabla\Psi(y)|}{(1+|y|)^2} \rho(t,y) dy dt <\infty,\quad \forall T>0. 
\end{equation}
For any fixed $T>0$, note that it holds for all $t\in[0, T]$ that
\begin{align*}
\int_{\R^d} \frac{{ 2\lambda} + { 2\eta} \rho^{m-1}(t, y) + |y \cdot\nabla\Psi(y)|}{(1+|y|)^2} \rho(t,y) dy 
&\le 2\lambda + 2\eta \|\rho^m(t,\cdot)\|^m_{L^m(\R^d)} + K\\
&\le (2\lambda+K) + 2\eta e^{(m-1)A T} \|\rho_0\|^m_{L^m(\R^d)},
\end{align*}
where the first inequality follows from \eqref{Psi' linear growth} and the second inequality is due to Proposition~\ref{prop:L^m norm bdd}. As the right-hand side is finite (as a consequence of $J_{\lambda,\eta}(\rho_0)<\infty$; recall Remark~\ref{rem:J}) and independent of $t\in [0,T]$, we conclude that \eqref{to show for SPP} holds. Now, consider the path space $\Omega := C([0,T];\R^d)$, its Borel $\sigma$-algebra $\mathcal F$, and  the canonical filtration $\cF_t:= \sigma(\omega_s: 0\le s\le t)$ for $t\ge 0$. For each $T\in\N$, we conclude from \cite[Theorem 1.1]{BRS21} that (i) there exists a Borel probability measure $\P_T$ on $(\Omega, \cF_T)$ that solves the martingale problem associated with $Y$ in \eqref{SDE mixed rho} on $[0,T]$; (ii) for any $t\in[0,T]$, $\P_T\circ (Y_t)^{-1}(A) = \int_A \rho(t,y)dy$ for all Borel $A\subseteq\R^d$. Note that (i) readily implies that the canonical process $Y_t(\omega):=\omega_t$ fulfills the dynamics \eqref{SDE mixed rho} for $t\in [0,T]$; see e.g., \cite[Proposition 5.4.6]{KS-book-91}. As (ii) implies that the density function $\theta(t) := \rho^{Y_t}$ exists and equals $\rho(t)$ $\mathcal L^d$-a.e.\ for all $t\in[0,T]$, we further conclude that $Y_t(\omega)=\omega_t$ satisfies \eqref{SDE mixed ell} for $t\in[0,T]$. Finally, by \cite[Theorem 1.3.5]{SV-book-06}, there is a unique probability measure $\P$ on $(\Omega,\cF)$ such that $\P = \P_T$ on $\cF_T$ for all $T\in\N$. Thus, $Y_t(\omega)=\omega_t$ satisfies \eqref{SDE mixed ell} for all $t\ge 0$ and is thus a solution to \eqref{SDE mixed}. 
\end{proof}

In fact, the uniform $L^m(\R^d)$-boundedness of $\{\rho^{Y_t}\}_{t\in[0,T]}$ (for each $T>0$) characterizes the time-marginal laws of any solution to SDE \eqref{SDE mixed}, as the next uniqueness result shows. 

\begin{theorem}\label{thm:uniqueness Y}
Fix $\lambda\in(0,e/2)$, $\eta>0$, and $m>1$. Suppose that $\Psi\in C^1(\R^d)$, \eqref{Psi' linear growth} holds, and $\nabla \Psi:\R^d\to\R^d$ is Lipschitz. For any $\rho_{0} \in \Dc_2(\mathbb{R}^{d})$ such that $J_{\lambda,\eta}(\rho_0)<\infty$, let $Y$ be a solution to \eqref{SDE mixed} such that $\theta(t) := \rho^{Y_t}$, $t\ge 0$, satisfies \eqref{C condition} and 
\begin{equation}\label{esssup L^m bdd}
\esssup_{t\in [0,T]} \|\theta(t)\|_{L^m(\R^d)}<\infty\quad \forall T>0. 
\end{equation}
Then, $\theta(t)=\rho(t)$ $\Lc^d$-a.e.\ for all $t\ge 0$, where $\rho:[0,\infty)\to \Dc_2(\R^d)$ is the solution to the Fokker-Planck equation \eqref{FPE mixed} obtained in Theorem~\ref{thm:sol. to FPE}.
\end{theorem}

\begin{proof}
We will show that $\theta:[0,\infty)\to\Dc(\R^d)$ is a solution to \eqref{FPE mixed} that belongs to $\mathcal C_{\rho_0}$ in \eqref{C_{rho_0}}, so that the desired result follows from Theorem~\ref{thm:uniqueness FPE}. As we already assume \eqref{C condition}, it remains to prove that $\theta(t)\in\Dc_2(\R^d)$ for all $t\ge 0$, 
$\theta(t)\Lc^d\to\theta(0)\Lc^d$ in $\Pc_2(\R^d)$ as $t\downarrow 0$, and $\theta$ is a solution to \eqref{FPE mixed}. To this end, observe that for any $0\le s< t$, if $X$ is a bounded process on $[s,t]$ (i.e., there is $M>0$ such that $|X_r(\omega)| \le M$ for all $r\in [s,t]$ and $\omega\in\Omega$), $Z_r := \sqrt{2\lambda +2\eta \theta^{m-1}(r,Y_r)} X_r$, $r\in[s,t]$, satisfies 
\begin{equation}\label{Z L^2}
\E\bigg[\int_{s}^t |Z_r|^2 dr\bigg] \le M^2 \int_{s}^t \left(2\lambda +2 \eta \int_{\R^d} \theta^{m}(r,y)\, dy\right)\, dr <\infty, 
\end{equation}
where the finiteness stems from \eqref{esssup L^m bdd}. 

For any $T>0$, consider $\tau_n := \inf\{t\ge 0: |Y_t|>n \}\wedge T$ for all $n\in\N$. As It\^{o}'s formula yields 
\begin{align*}
 d |Y_t|^2  = \left(-2 Y_t\cdot \nabla\Psi(Y_t) + 2\lambda+2\eta(\theta(t,Y_t))^{m-1}\right)\, dt + 2  \sqrt{2\lambda+2\eta(\theta(t,Y_t))^{m-1}}Y_t\cdot dB_t,
\end{align*}
we deduce from \eqref{Psi' linear growth} and \eqref{Z L^2} that for any $t\in[0,T]$, 
\[
\E|Y_{t\wedge \tau_n}|^2 \le \E|Y_0|^2 + \int_0^{t\wedge \tau_n} \left(2C_K (1+\E|Y_r|^2) + C_T\right)\, dr ,\quad n\in\N,
 \]
with $C_K>0$ depending on $K>0$ in \eqref{Psi' linear growth} and $C_T := 2\lambda+2\eta \esssup_{t\in [0,T]}\|\theta(t)\|^m_{L^m(\R^d)}<\infty$. As $n\to\infty$, Fatou's lemma implies $\E|Y_{t}|^2 \le \E|Y_0|^2 +(2C_K+C_T)t + \int_0^{t} 2C_K \E|Y_s|^2 ds$ for all $t\in [0,T]$. With $\theta(0)=\rho_0\in\cD_2(\R^d)$ (i.e., $\E|Y_0|^2<\infty$), Gr\"{o}nwall's inequality yields 
\begin{equation}\label{E|Y_t|^2 bdd}
\E|Y_{t}|^2 \le \big(\E|Y_0|^2 +(2C_K+C_T)T\big) e^{2C_KT}<\infty\quad \forall t\in[0,T]. 
\end{equation}
As $T>0$ is arbitrary, we get $\theta(t)\in \Dc_2(\R^d)$ for all $t\ge 0$. 

On the other hand, for any $0\le s< t< T$, using the dynamics in \eqref{SDE mixed rho} directly, we get 
\begin{align*}
|Y_t-Y_s|^2 &= \bigg|-\int_{s}^t \nabla\Psi(Y_r)\, dr +\int_s^t   \sqrt{{ 2\lambda} + { 2\eta} \theta^{m-1}(r, Y_{r})}\, dB_r\bigg|^2\\
&\le 2 (t-s)\int_{s}^t |\nabla\Psi(Y_r)|^2\, dr + 2 \bigg|\int_s^t   \sqrt{{ 2\lambda} + { 2\eta} \theta^{m-1}(r, Y_{r})}\, dB_r\bigg|^2,
\end{align*}
where the inequality follows from $(a+b)^2\le 2(a^2+b^2)$ for all $a,b\in\R$ and H\"{o}lder's inequality. Note that  because of \eqref{Z L^2}, the process $u\mapsto \int_s^u   \sqrt{{ 2\lambda} + { 2\eta}\theta^{m-1}(r, Y_{r})}\, dB_r$ is a martingale. Hence, by \eqref{Psi' linear growth} and It\^{o}'s isometry, the previous inequality implies  
\begin{align*}
\E |Y_t-Y_s|^2 &\le 2(t-s)\int_s^t C_K(1 +\E|Y_r|^2)\, dr + 2 \E\bigg[\int_s^t { 2\lambda} + { 2\eta} \theta^{m-1}(r, Y_{r})\, dr\bigg]\\
&\le 2\Big[C_K \Big(1+ \E|Y_0|^2+ (2C_K+C_T)T\Big) e^{2C_K T} T+C_T\Big](t-s),
\end{align*}
where the second inequality stems from \eqref{E|Y_t|^2 bdd}. As $T>0$ is arbitrary, $t\mapsto Y_t$ is continuous on $[0,\infty)$ in $L^2(\Omega,\mathcal F,\P)$, so that $t\mapsto\theta(t)\Lc^d$ is continuous on $[0,\infty)$ in $\Pc_2(\R^d)$. We thus obtain $\theta(t)\Lc^d\to\theta(0)\Lc^d$ in $\Pc_2(\R^d)$ as $t\downarrow 0$ and the weak-$*$ continuity of $t\mapsto\theta(t)$ on $[0,\infty)$ in $L^1(\R^d)$. 


Now, for any $\varphi\in C^{\infty}_c((0,\infty)\times\R^d)$, we can choose $T>0$ large enough such that $\varphi(1/T,\cdot) = \varphi(T,\cdot) \equiv 0$. By applying It\^{o}'s formula to $\varphi(Y_t)$, we get 
\begin{align}
0 = \varphi(T,Y_T) - \varphi(1/T, Y_{1/T}) &= \int_{1/T}^T \Big(\varphi_t -\nabla\varphi\cdot \nabla\Psi + (\lambda + \eta \theta^{m-1})\Delta \varphi\Big)(t,Y_t) dt\notag\\
& \hspace{0.75in}+  \int_{1/T}^T \sqrt{2\lambda +2\eta \theta^{m-1}(Y_t)} \nabla\varphi(t, Y_t)\cdot dB_t.\label{to get FPE}
\end{align}
We claim that, by taking expectation on both sides, one gets 
\begin{equation}\label{to get FPE'}
0 = \int_{1/T}^T \int_{\R^d} \Big(\theta\varphi_t -\theta\nabla\varphi\cdot \nabla\Psi + (\lambda \theta + \eta \theta^{m})\Delta \varphi\Big)(t,y)\, dy\, dt.  
\end{equation}
When one takes expectation of the first term on the right-hand side of \eqref{to get FPE}, Fubini's theorem can be applied (thanks to \eqref{Psi' linear growth}, \eqref{E|Y_t|^2 bdd}, and \eqref{esssup L^m bdd}), which yields the right-hand side of \eqref{to get FPE'}. Thanks to \eqref{Z L^2}, the process $u\mapsto \int_{1/T}^u   \sqrt{{ 2\lambda} + { 2\eta}\theta^{m-1}(r, Y_{r})}\nabla\varphi(t, Y_t)\cdot dB_r$ is a martingale, so that the expectation of the second term on the right-hand side of \eqref{to get FPE} is zero. This thus establishes \eqref{to get FPE'}. As $L_F(s) = \lambda s + \eta s^m$ (by \eqref{L_G} and \eqref{F}) and $L_F(\theta(t))\in W^{1,1}_{\operatorname{loc}}(\R^d)$ for all $t\ge 0$ (recall \eqref{C condition}), we can apply integration by parts to \eqref{to get FPE'}, turning $(\lambda \theta + \eta \theta^{m})\Delta \varphi$ into $\nabla(\lambda \theta + \eta \theta^{m})\cdot \nabla\varphi$. With this, taking $T\to\infty$ in \eqref{to get FPE'} yields \eqref{distri. sense} (with $\rho=\theta$ therein). This, along with the weak-$*$ continuity of $t\mapsto\theta(t)$ on $[0,\infty)$ in $L^1(\R^d)$, shows that $\theta$ is a solution to \eqref{FPE mixed} (cf.\ Definition~\ref{def:sol. to FPE}). 
\end{proof}



\section{Minimizer of $J_{\lambda,\eta}$}\label{sec:minimizer}
This section aims for an explicit characterization of the minimizer of $J_{\lambda,\eta}$ in \eqref{J}. 
For any $\lambda,\eta>0$ and $m>1$, recall from \eqref{L_G} and \eqref{F} that $L_F(s) = \lambda s + \eta s^{m}$. Following \cite[Section 3]{Carrillo01}, we define
\begin{eqnarray}
    h(z) &:=& \int_1^z  \frac{L_F'(s)}{s} ds = {\lambda}\ln z + \frac{\eta m}{m-1} (z^{m-1} - 1),\quad z>0. \label{h} 
\end{eqnarray}
As $h$ in \eqref{h} is strictly increasing with $h(0+) = -\infty$ and $h(\infty)=\infty$, for any $C \in \mathbb{R}$, 
\begin{equation}\label{U(y;C)}
    U(y;C) := h^{-1}(C - \Psi(y))>0, \quad y \in \R^d,
\end{equation}
is well-defined. 
As $h(U(y;C)) = C - \Psi(y)$ for all $y\in\R^d$ by definition, we deduce from \eqref{h} that $\rho(\cdot):= U(\cdot;C):\R^d\to (0,\infty)$ is a solution to the equation 
\begin{equation} \label{Eq: Equilibrium identity}
    {\lambda}\ln \rho(y) + \frac{\eta m}{m-1} \rho^{m-1}(y) = \bigg(\frac{\eta m}{m-1} + C \bigg) - \Psi(y),\quad y\in\R^d. 
\end{equation}

\begin{definition}\label{def:W_0}
Let ${\bm{{\bf W_0}}}: [0,\infty)\to[0,\infty)$ be the inverse function of $y\mapsto u(y):= ye^y$ for $y\ge 0$. That is, for any $x,y\ge 0$, ${\bf W_0}(x)=y$ if and only if $x=y e^y$. 
\end{definition}

\begin{remark}\label{rem:W_0}
${\bf W_0}$ is the principal branch of the (complex-valued) Lambert {\bf W} function; see e.g., \cite{Corless1996, Hoorfar2008}. Note that (i) 
as a consequence of Definition~\ref{def:W_0}, for any $c,x,y>0$, $x=y e^{cy}$ if and only if $y= \frac1c {\bf W_0}(cx)$; (ii) by \cite[Theorem 2.3]{Hoorfar2008}, for any $x>0$ and $c>1/e$, ${\bf W_0}(x) \le \ln(\frac{x+c}{1+\ln c})$.    
\end{remark}

\begin{lemma}\label{lem:U}
Fix $\lambda,\eta>0$ and $m>1$. For any $C\in\R$, $U(\cdot;C):\R^d\to (0,\infty)$ in \eqref{U(y;C)} is the unique positive solution to \eqref{Eq: Equilibrium identity} and it has the explicit formula
    \begin{equation} \label{U formula}
        U(y;C)= \left( \frac{\lambda}{\eta m} {{\bm{{\bf W_0}}}} \left(\frac{\eta m }{\lambda} e^{-\frac{m-1}{\lambda} \left\{\Psi(y)-(\frac{\eta m}{m-1} + C)\right\}} \right) \right)^{\frac{1}{m-1}},\quad y\in\R^d.
    \end{equation}
If $\Psi:\R^d\to [0,\infty)$ grows at least logarithmically at infinity, i.e.,   
\begin{equation}\label{>=log growth}
\exists R, \delta >0\quad \hbox{s.t.}\quad \Psi(y) \ge ( \lambda d+\delta) \ln(|y|)\quad \forall |y|\ge R, 
\end{equation}
then $U(\cdot;C)\in L^k(\R^d)$ for all $k\ge 1$. Moreover, if \eqref{>=log growth} is strengthened to 
\begin{equation}\label{>=log growth'}
\exists R,\delta>0\quad \hbox{s.t.}\quad \Psi(y) \ge \big(\lambda (d+2)+\delta\big) \ln(|y|)\quad \forall |y|\ge R, 
\end{equation}
then we additionally have $\int_{\R^d} |y|^2 U(y;C)dy<\infty$. 
\end{lemma}

\begin{proof}
First, \eqref{Eq: Equilibrium identity} can be rearranged into $-\frac1\lambda \big\{\Psi(y)-(\frac{\eta m}{m-1} + C)\big\} = \ln \big( \rho(y) e^{\frac{\eta m}{\lambda(m-1)} \rho^{m-1}(y)} \big)$ for $y\in\R^d$. By plugging both sides into the exponential function and raising them to the power of $m-1$, we get the equivalent relation $e^{-\frac{m-1}{\lambda} \{\Psi(y)-(\frac{\eta m}{m-1} + C)\}} = \rho^{m-1}(y) e^{\frac{\eta m}{\lambda} \rho^{m-1}(y)}$ for $y\in\R^d$. By Remark~\ref{rem:W_0} (i), this is equivalent to $ \rho^{m-1}(y) = \frac{\lambda}{\eta m} {\bf W_0}\big(\frac{\eta m}{\lambda} e^{-\frac{m-1}{\lambda} \{\Psi(y)-(\frac{\eta m}{m-1} + C)\}}\big)$ for $y\in\R^d$. That is, $\bar\rho(y) := \big(\frac{\lambda}{\eta m} {\bf W_0}\big(\frac{\eta m}{\lambda} e^{-\frac{m-1}{\lambda} \{\Psi(y)-(\frac{\eta m}{m-1} + C)\}}\big)\big)^{1/(m-1)}$, $y\in\R^d$, is the unique positive solution to \eqref{Eq: Equilibrium identity}. The discussion below \eqref{U(y;C)} then entails $U(y;C)=\bar\rho(y)$ for $y\in\R^d$, which proves \eqref{U formula}. 

By taking $c = 1$ in Remark~\ref{rem:W_0} (ii) and using the fact $\ln(v+1) \leq v$ for $v > 0$, we have
    \begin{align} 
        {\bf W_0} \left( \frac{\eta m}{\lambda} e^{-\frac{m-1}{\lambda} \{\Psi(y)-(\frac{\eta m}{m-1} + C)\}} \right) &\leq \ln \left( \frac{\eta m}{\lambda} e^{-\frac{m-1}{\lambda}\{\Psi(y)-(\frac{\eta m}{m-1} + C)\}} + 1 \right)\notag\\
        & \leq \frac{\eta m}{\lambda} e^{-\frac{m-1}{\lambda} \{\Psi(y)-(\frac{\eta m}{m-1} + C)\}}.\label{W_0<exp}
    \end{align}
Now, for any $k\ge 1$, observe that
    \begin{equation} \label{Eq: Upper bound on rho m-1}
        U(y;C)^{k} = \big(U(y;C)^{m-1}\big)^{\frac{k}{m-1}} \leq e^{-\frac{k}{\lambda} \{\Psi(y)-(\frac{\eta m}{m-1} + C)\}}, 
    \end{equation}
    where the inequality follows from \eqref{U formula} and \eqref{W_0<exp}. As a result,
        \begin{align}\label{L^k}
            \int_{\mathbb{R}^{d}}   U(y;C)^{k} \, dy &\leq  e^{\frac{k}{\lambda}(\frac{\eta m}{m-1} + C)}  \int_{\mathbb{R}^{d}} e^{-\frac{k}{\lambda} \Psi(y)} \, dy \notag \\
            &\le e^{\frac{k}{\lambda}(\frac{\eta m}{m-1} + C)}  \bigg(\int_{|y| \leq R} e^{-\frac{k}{\lambda} \Psi(y)} \, dy + \int_{|y| > R} |y|^{-k (d + \delta/\lambda)}\, dy \bigg),
        \end{align}
    where the second inequality stems from \eqref{>=log growth}. On the right-hand side of \eqref{L^k}, the first integral is finite because $\Psi$ is bounded from below, while the second integral is bounded from above by 
\begin{equation}\label{polar}
\int_{|y| > R} |y|^{-(d + \delta/\lambda)}\, dy = K \int_{R}^\infty r^{-(d + \delta/\lambda)} r^{d-1}\, dr =  K \int_{R}^\infty r^{-(1 + \delta/\lambda)}\, dr = K \frac{\lambda}{\delta} R^{-\delta/\lambda},
\end{equation}
where the second equality follows by using the $d$-dimensional spherical coordinates (which yields the constant multiple $K>0$). Hence, we conclude from \eqref{L^k} that $U(y;C)\in L^k(\R^d)$. 

Under \eqref{>=log growth'}, by using \eqref{Eq: Upper bound on rho m-1} (with $k=1$) as in \eqref{L^k}, we get
\begin{align*}
\int_{\R^d} |y|^2 U(y;C)dy &\le  e^{\frac{1}{\lambda}(\frac{\eta m}{m-1} + C)}   \int_{\R^d} |y|^2 e^{-\frac{1}{\lambda}\Psi(y)}dy\\
&\le e^{\frac{1}{\lambda}(\frac{\eta m}{m-1} + C)}  \bigg(\int_{|y| \leq R} |y|^2 e^{-\frac{1}{\lambda} \Psi(y)} \, dy + \int_{|y| > R} |y|^{-(d+\delta/\lambda)} \, dy \bigg)<\infty,
\end{align*}
where the finiteness stems from $\Psi$ being bounded from below and \eqref{polar}. 
\end{proof}

\begin{corollary}\label{coro:rho_infty}
Fix $\lambda,\eta>0$ and $m>1$. Assume \eqref{>=log growth}. Then, there is a unique $C^*>0$ such that $\int_{\R^d} U(y;C^*)dy =1$. Hence, the function
\begin{equation}\label{rho_infty}
\rho_\infty(y) := U(y;C^*),\quad y\in\R^d,
\end{equation}
belongs to $\Dc(\R^d)\cap L^k(\R^d)$ for all $k>1$. If \eqref{>=log growth} is strengthened to \eqref{>=log growth'}, then $\rho_\infty\in \Dc_2(\R^d)$ and $\rho_\infty\ln(\rho_\infty)\in L^1(\R^d)$. If we further assume \eqref{Psi' linear growth}, then $\Psi\rho_\infty\in L^1(\R^d)$ and thus $J_{\lambda,\eta}(\rho_\infty)<\infty$. 
\end{corollary}

\begin{proof}
By Definition~\ref{def:W_0}, ${\bf W_0}$ is strictly increasing and continuous with $\lim_{x\downarrow 0}{\bf W_0(x)}=0$ and $\lim_{x\to\infty}{\bf W_0}(x)=\infty$. In view of \eqref{U formula}, this implies that for any fixed $y\in\R^d$, $C\mapsto U(y;C)$ is strictly increasing and continuous with $\lim_{C\to-\infty}U(y;C) = 0$ and $\lim_{C\to\infty}U(y;C)=\infty$. This immediately shows that $C\mapsto M(C):=\int_{\R^d} U(y;C) dy$ is strictly increasing. As $U(y;C)\ge 0$ by definition, $C\mapsto U(y;C)$ is strictly increasing, and $U(\cdot;C)\in L^1(\R^d)$ for all $C\in\R$ (by Lemma~\ref{lem:U}), we conclude from the dominated convergence theorem that $C\mapsto M(C)$ is continuous. Also, the monotone convergence theorem yields $\lim_{C\to-\infty}M(C) = 0$ and $\lim_{C\to\infty}M(C)=\infty$. All these properties of $M$ ensure the existence of a unique $ C^*\in\R$ such that $M(C^*)=1$. With $M(C^*)=1$ and $\rho_\infty\in L^k(\R^d)$ for all $k\ge 1$ (by Lemma~\ref{lem:U}), we conclude $\rho_\infty\in\Dc(\R^d)\cap L^k(\R^d)$ for all $k>1$. 

Under \eqref{>=log growth'}, Lemma~\ref{lem:U} asserts $\int_{\R^d} |y|^2 \rho_\infty(y) dy<\infty$, so that $\rho_\infty\in \Dc_2(\R^d)$. With $\rho_\infty\in \Dc_2(\R^d)$, $(\rho_\infty\ln\rho_\infty)^-$ is integrable; recall Remark~\ref{rem:J}. As $\ln y\le y$ for all $y>0$, $(\rho_\infty\ln\rho_\infty)^+$ is also integrable because $\int_{\R^d} (\rho_\infty\ln\rho_\infty)^+(y) dy \le \int_{\R^d} \rho_\infty^2(y)dy<\infty$, where the finiteness follows from $\rho_\infty\in L^2(\R^d)$. Hence, we conclude $\rho_\infty\ln(\rho_\infty)\in L^1(\R^d)$. 
Finally, under \eqref{Psi' linear growth}, there exists $C_K>0$ (depending on $K>0$ in \eqref{Psi' linear growth}) such that $|\Psi(y)|\le C_K(1+|y|^2)$ for all $y\in\R^d$. Hence,
$\int_{\R^d} |\Psi(y) \rho_\infty(y)| \, dy \le  \int_{\R^d} C_K(1+|y|^2) \rho_\infty(y) \, dy <\infty$, where the finiteness results from $\rho_\infty\in \Dc_2(\R^d)$. That is, $\Psi\rho_\infty\in L^1(\R^d)$. This, along with $\rho_\infty\ln(\rho_\infty)\in L^1(\R^d)$ and $\rho_\infty\in L^m(\R^d)$, implies $J_{\lambda,\eta}(\rho_\infty)<\infty$. 
\end{proof}

\begin{theorem}\label{thm:minimizer of J}
Fix $\lambda,\eta>0$ and $m>1$. Assume \eqref{>=log growth'} and \eqref{Psi' linear growth}. Then, $\rho_{\infty}$ in \eqref{rho_infty} is the unique minimizer of $J_{\lambda,\eta}:\mathcal D_2(\R^d)\to\R$ in \eqref{J}. 
\end{theorem}

\begin{proof}
By Corollary~\ref{coro:rho_infty}, $\rho_\infty\in\Dc_2(\R^d)$ and $J_{\lambda,\eta}(\rho_\infty)<\infty$. 
Recall $h:(0,\infty)\to\R$ in \eqref{h} 
and define $\Phi:[0,\infty)\to\R$ by $\Phi(z) := \int_0^z h(s) ds = {\lambda} z \ln z  + \frac{\eta}{m-1} z^{m}  - \big(\lambda+\frac{\eta m}{m-1}\big)z$ for $z\ge 0$.
Then, consider $E:\mathcal D_2(\R^d)\to\R\cup\{\infty\}$ defined by 
$E(\rho) :=  \int_{\mathbb{R}^{d}}  \Psi(y)\rho(y) + \Phi(\rho(y))\, dy$ for $\rho\in\Dc_2(\R^d).$ 
In view of \eqref{J}, a direct calculation shows that 
$E(\rho) = J_{\lambda,\eta}(\rho) -  \big(\lambda+\frac{\eta m}{m-1}\big)\ \forall \rho\in\Dc_2(\R^d).$ 
As $J_{\lambda,\eta}(\rho_\infty)<\infty$, we have $E(\rho_\infty)<\infty$. Under this condition, \cite[Lemma 6]{Carrillo01} asserts that $\rho_\infty$ is the unique minimizer of $E$, whence also the unique minimizer of $J_{\lambda,\eta}$. 
\end{proof}


\section{Convergence to the Minimizer of $J_{\lambda,\eta}$}\label{sec:convergence}
In this section, we will prove the long-time convergence of $\{\rho^{Y_t}\}_{t\ge 0}$, the density flow induced by a solution $Y$ to SDE \eqref{SDE mixed}, to the minimizer $\rho_\infty$ of the functional $J_{\lambda,\eta}$ in \eqref{J}. As $\{\rho^{Y_t}\}_{t\ge 0}$ coincides with the solution $\rho:[0,\infty)\to\Dc_2(\R^d)$ to the Fokker-Planck equation \eqref{FPE mixed} obtained in Theorem~\ref{thm:sol. to FPE} (recall Theorem~\ref{thm:uniqueness Y}), we will mainly focus on analyzing the long-time convergence of $\rho(t)$. 

First, we note that by the construction of $\rho$, $t\mapsto J_{\lambda,\eta}(\rho(t))$ is non-decreasing. 

\begin{lemma} \label{lem:J decreasing}
Fix $\lambda\in(0,e/2)$, $\eta>0$, and $m>1$. Assume  $\Psi\in C^1(\R^d)$ and \eqref{Psi' linear growth}. For any $\rho_0\in \mathcal D_2(\R^d)$ such that $J_{\lambda,\eta}(\rho_0) <\infty$, let $\rho:[0,\infty)\to\Dc_2(\R^d)$ be the solution to \eqref{FPE mixed} obtained in Theorem~\ref{thm:sol. to FPE}. Then, 
$J_{\lambda, \eta}(\rho(t)) \leq J_{\lambda, \eta}(\rho_{0})$ for all $t\ge 0.$
If additionally $\nabla\Psi$ is Lipschitz, then $t\mapsto J_{\lambda, \eta}(\rho(t))$ is non-increasing.  
\end{lemma}

\begin{proof}
Take $\mu_0 = \rho_0\Lc^d\in\Pc_2(\R^d)$. For any $\tau>0$, recall the minimizing movement scheme \eqref{Theta}-\eqref{M in argmin}. As shown in the proof of Proposition~\ref{prop:L^m norm bdd}, $\Jc_{\lambda,\eta}(M^n_\tau)= J_{\lambda,\eta}(\rho^n_\tau)\le J_{\lambda,\eta}(\rho_0)$ for all $n\in\N$; see \eqref{J^n<=J^n-1}. By \eqref{overline M}, $\Jc_{\lambda,\eta}(\overline M_\tau(t))\le J_{\lambda,\eta}(\rho_0)$ for all $t\ge 0$. As $\tau\downarrow 0$, recall from Remark~\ref{rem:Thm 11.1.6} that $\overline M_\tau(t)\to \rho(t)\Lc^d$ weakly in $\Pc(\R^d)$ for all $t>0$. Hence, by the lower semicontinuity of $\Jc_{\lambda,\eta}$ w.r.t.\ weak convergence in $\Pc(\R^d)$ (Remark~\ref{rem:LSC of Jc}), we get $\Jc_{\lambda,\eta}(\rho(t)\Lc^d)\le J_{\lambda,\eta}(\rho_0)$, i.e., $J_{\lambda, \eta}(\rho(t)) \leq J_{\lambda, \eta}(\rho_{0})$, for all $t\ge 0$. If $\nabla\Psi$ is Lipschitz, Corollary~\ref{coro:flow property} implies that for any $0\le s<t$, $\rho(t) = S(t)[\rho_0] = S(t-s) [S(s)[\rho_0]]= S(t-s)[\rho(s)]$ $\Lc^d$-a.e. Thus, $J_{\lambda,\eta}(\rho(t)) = J_{\lambda,\eta}(S(t-s)[\rho(s)]) \le J_{\lambda,\eta}(\rho(s))$. 
\end{proof}

In fact, how $J_{\lambda,\eta}(\rho(t))$ decreases over time can be characterized much more explicitly. To this end, we need to first show that $\Jc_{\lambda,\eta}$ in \eqref{Jc} is {\it regular}, following \cite[Definition 10.1.4]{Ambrosio08}. 

\begin{definition}
\label{def:regular}
Consider a lower semicontinuous $\phi: \mathcal{P}_{2}(\R^d) \to (-\infty,\infty]$ such that \eqref{D(phi)} holds. 
We say $\phi$ is regular if whenever $\mu_n\in \Pc_2(\R^d)$ and $\xi_n\in \partial \phi(\mu_n)$, $n\in\N$, satisfy
(i) $\mu_{n} \to \mu$ in $\mathcal{P}_{2}(\R^d)$ for some $\mu\in\Pc_2(\R^d)$, (ii) $\phi(\mu_n)\to \ell$ for some $\ell\in\R$, (iii) $\sup_{n\in\N} \|\xi_{n}\|_{L^{2}(\R^d,\mu_{n})} < \infty$, and (iv) ``$\xi_n$ converges weakly to some $\xi\in L^2(\R^d,\mu)$'' in the following sense:
\begin{equation}\label{weak converg. to xi}
\lim_{n\to\infty} \int_{\R^d} \varphi(y) \xi_n(y) d\mu_n(y) =  \int_{\R^d} \varphi(y) \xi(y) d\mu(y)\quad \forall \varphi\in C^\infty_c(\R^d), 
\end{equation}
then  $\xi \in \partial \phi(\mu)$ and $\phi(\mu)=\ell$.
 \end{definition}


\begin{lemma}\label{lem:Jc regular}
Fix $\lambda\in(0,e/2)$, $\eta>0$, and $m>1$. Assume $\Psi\in C^1(\R^d)$ and \eqref{Psi' linear growth}. Then, $\Jc_{\lambda,\eta}:\Pc_2(\R^d)\to (-\infty,\infty]$ in \eqref{Jc} is regular. 
\end{lemma}

\begin{proof}
By \eqref{Psi' linear growth}, there exists $C_K>0$ (depending on $K>0$ in \eqref{Psi' linear growth}) such that $\Psi(y)\le C_K(1+|y|^2)$ for all $y\in\R^d$. Hence,
$\int_{\R^d} \Psi \rho\, dy <\infty$ for all $\rho\in\Dc_2(\R^d)$. Also note that $\int_{\R^d} (\rho\ln\rho)^+dy \le \int_{\R^d} \rho^2 dy$ for all $\rho\in \Dc(\R^d)$. We then deduce from Remark~\ref{rem:J} that $J_{\lambda,\eta}(\rho)<\infty$ if $\rho\in \Dc_2(\R^d)\cap L^{m\vee 2}(\R^d)$. Thus, in view of \eqref{Jc}, $D(\Jc_{\lambda,\eta})\neq\emptyset$ and $\mu\in D(\Jc_{\lambda,\eta})$ implies $\mu=\rho\Lc^d$ for some $\rho\in\Dc_2(\R^d)$. Also recall from Remark~\ref{rem:LSC of Jc} that $\Jc_{\lambda,\eta}$ is lower semicontinuous. 

Take $\mu_n\in \Pc_2(\R^d)$ and $\xi_n\in \partial \Jc_{\lambda,\eta}(\mu_n)$, $n\in\N$, that satisy conditions (i)-(iv) in Definition~\ref{def:regular}. By Definition~\ref{def:subdifferential}, $\mu_n$ must lie in $D(\Jc_{\lambda,\eta})$, such that $\mu_n = \rho_n\Lc^d$ for some $\rho_n\in\Dc_2(\R^d)$. Then, by Corollary~\ref{coro:subdiff.=}, $\xi_n\in \partial \Jc_{\lambda,\eta}(\rho_n\Lc^d)$ readily implies 
$
\xi_n = \nabla\Psi+{\nabla L_F(\rho_n)}/{\rho_n}$ \hbox{and} $\partial\Ic_F(\rho_n\Lc^d) = \{\nabla L_F(\rho_n)/\rho_n\}. 
$
As $\mu_n\to\mu$ in $\Pc_2(\R^d)$ and $\nabla\Psi$ is continuous and satisfies \eqref{Psi' linear growth}, by \cite[Definition 6.8 (iv) and Theorem 6.9]{Villani-book-09}, we have (a) $\int_{\R^d} \nabla\Psi d\mu_n\to \int_{\R^d} \nabla\Psi d\mu<\infty$; (b) $\int_{\R^d} |\nabla\Psi|^2 d\mu_n\to \int_{\R^d} |\nabla\Psi|^2 d\mu<\infty$. This has several implications: (i) Recall from \eqref{J=G+I_F} that $\mathcal J_{\lambda,\eta} (\bar\mu) = \mathcal G(\bar\mu) +  \mathcal I_F(\bar\mu) = \int_{\R^d} \nabla\Psi d\bar\mu+ \mathcal I_F(\bar\mu)$ for $\bar\mu\in\Pc_2(\R^d)$. Hence, by $\mathcal J_{\lambda,\eta}(\mu_n)\to \ell$ and (a), we get $\mathcal I_F(\mu_n)\to \ell - \int_{\R^d} \nabla\Psi d\mu\in\R$. (ii) By (b), $\sup_{n\in\N}\|\nabla\Psi\|_{L^2(\R^d,\mu_n)}<\infty$. Since ${\nabla L_F(\rho_n)}/{\rho_n} = \xi_n -\nabla\Psi$ and $\sup_{n\in\N} \|\xi_{n}\|_{L^{2}(\R^d,\mu_{n})} < \infty$, this implies $\sup_{n\in\N} \|\nabla L_F(\rho_n)/{\rho_n}\|_{L^{2}(\R^d,\mu_{n})} < \infty$. (iii) For any $\varphi\in C^\infty_c(\R^d)$, as $\varphi \nabla\Psi$ is continuous and bounded, $\int_{\R^d} \varphi \nabla\Psi d\mu_n \to \int_{\R^d} \varphi\nabla\Psi d\mu$. This, along with \eqref{weak converg. to xi}, implies that ${\nabla L_F(\rho_n)}/{\rho_n} = \xi_n -\nabla\Psi$ converges weakly to $\xi-\nabla\Psi\in L^2(\R^d,\mu)$ (i.e., \eqref{weak converg. to xi} is fulfilled with $\xi_n$ and $\xi$ therein replaced by ${\nabla L_F(\rho_n)}/{\rho_n}$ and $\xi-\nabla\Psi$, respectively). Now, in view of \eqref{Ic} and \eqref{F}, $\Ic_F:\Pc_2(\R^d)\to (-\infty,\infty]$ is geodesically convex (by \cite[Proposition 9.3.9]{Ambrosio08}) and therefore regular due to \cite[Lemma 10.1.3]{Ambrosio08}. Hence, the properties (i), (ii), and (iii) above imply $\xi-\nabla\Psi\in \partial\Ic_F(\mu)$ and $\Ic_F(\mu) = \ell - \int_{\R^d} \nabla\Psi d\mu\in\R$. 
As $\partial \mathcal G(\mu) = \{\nabla\Psi\}$ (see the proof of  Corollary~\ref{coro:subdiff.=}), we deduce from $\mathcal J_{\lambda,\eta} (\mu) = \mathcal G(\mu) +  \mathcal I_F(\mu)$ and Definition~\ref{def:subdifferential} that $\xi = \nabla\Psi + (\xi-\nabla\Psi)\in \partial \Jc_{\lambda,\eta}(\mu)$ and
$
\mathcal J_{\lambda,\eta} (\mu) = \int_{\R^d} \nabla\Psi d\mu + (\ell - \int_{\R^d} \nabla\Psi d\mu) = \ell.
$
We therefore conclude that $\Jc_{\lambda,,\eta}$ is regular. 
\end{proof}

As $\Jc_{\lambda,\eta}$ is regular, we can invoke a ``{chain rule}'' for absolutely continuous curves in $\Pc_2(\R^d)$. 

\begin{proposition}\label{prop:J'(t)}
Fix $\lambda\in(0,e/2)$, $\eta>0$, and $m>1$. Suppose that $\Psi\in C^1(\R^d)$, \eqref{Psi' linear growth} holds, and $\nabla\Psi$ is Lipschitz. For any $\rho_0\in \mathcal D_2(\R^d)$ such that $J_{\lambda,\eta}(\rho_0) <\infty$, let $\rho:[0,\infty)\to\Dc_2(\R^d)$ be the solution to \eqref{FPE mixed} obtained in Theorem~\ref{thm:sol. to FPE}. 
Then, 
        $\frac{{d}}{dt} J_{\lambda, \eta}(\rho(t)) = - \|v_\rho(t,\cdot)\|_{L^{2}(\R^d,\rho(t)\Lc^d)}^{2}$ for $\mathcal{L}^1$-a.e.\ $t > 0,$
where $v_\rho(t,y)$ is given by \eqref{v}. Hence,
    \begin{equation} \label{J<=J-}
        J_{\lambda, \eta}(\rho(t))  \leq J_{\lambda, \eta}(\rho_{0}) - \int_{0}^{t} \|v_\rho(r,\cdot)\|_{L^{2}(\mathbb{R}^{d}; \rho(r)\Lc^d)}^{2} \, dr, \quad\forall t \ge 0.
    \end{equation}
\end{proposition}

\begin{proof}
Consider $\mu_t := \rho(t)\Lc^d\in\Pc_2(\R^d)$ for $t\ge0$. By Theorem~\ref{thm:sol. to FPE}, $\{\mu_t\}_{t\ge 0}\in AC^2_{\operatorname{loc}}([0,\infty);\Pc_2(\R^d))$ and $\partial\mathcal J_{\lambda,\eta}(\mu_t) = \{-v_\rho(t,\cdot)\}$ for $\Lc^1$-a.e.\ $t>0$. Also note that \eqref{FPE mixed} can be equivalently written as $\partial_{t} \mu_{t} + \nabla \cdot (v_\rho \mu_{t}) = 0$. Now, for any $T>0$, we have $\{\mu_t\}_{t\in (0,T)}\in AC((0,T);\Pc_2(\R^d))$. As $\Jc_{\lambda, \eta}$ is regular (by Lemma~\ref{lem:Jc regular}), we may apply the chain rule in \cite[Section 10.1.2, part E]{Ambrosio08} to $t\mapsto\Jc_{\lambda, \eta}(\mu_t)$ at any $t\in(0,T)$ where the next three conditions hold: (a) The local slope  
$
|\partial \Jc_{\lambda, \eta}|(\mu_t) := \limsup_{\nu\to\mu_t} \frac{(\Jc_{\lambda, \eta}(\mu_t) - \Jc_{\lambda, \eta}(\nu))^+}{\Wc_2(\mu_t,\nu)}
$
is finite; (b) $\Jc_{\lambda,\eta}(\mu_t)$ is approximately differentiable at $t$; (c) \cite[(8.4.6)]{Ambrosio08} is satisfied at $t$. For $\Lc^1$-a.e.\ $t\in (0,T)$, as $\partial\mathcal J_{\lambda,\eta}(\mu_t) = \{-v_\rho(t,\cdot)\}$, by using \eqref{subdifferential'} and H\"{o}lder's inequality, $|\partial \Jc_{\lambda, \eta}|(\mu_t)\le \|v_\rho(t,\cdot)\|^2_{L^2(\R^d,\mu_t)}<\infty$. Also, as $t\mapsto\Jc_{\lambda,\eta}(\mu_t)=J_{\lambda,\eta}(\rho(t))$ is non-increasing (by Lemma~\ref{lem:J decreasing}), $\Jc_{\lambda,\eta}(\mu_t)$ is $\Lc^1$-a.e.\ differentiable (and thus approximately differentiable). Finally, for $\Lc^1$-a.e.\ $t\in (0,T)$, as $\partial\mathcal J_{\lambda,\eta}(\mu_t) = \{-v_\rho(t,\cdot)\}$ readily implies $v_\rho(t,\cdot)\in \operatorname{Tan}_{\mu_t}\Pc_2(\R^d)$ (recall Remark~\ref{rem:in Tan}), \cite[Proposition 8.4.6]{Ambrosio08} asserts that  \cite[(8.4.6)]{Ambrosio08} is fulfilled. With (a), (b), and (c) all satisfied $\Lc^1$-a.e.\ on $(0,T)$, the chain rule (specifically \cite[(10.1.16)]{Ambrosio08}) gives $\frac{{d}}{dt} \Jc_{\lambda, \eta}(\mu_t) = - \|v_\rho(t,\cdot)\|_{L^{2}(\R^d,\mu_t)}^{2}$ for $\Lc^1$-a.e.\ $t\in(0,T)$. As $T>0$ is arbitrary, the relation holds for $\Lc^1$-a.e.\ $t>0$. Thanks again to $t\mapsto J_{\lambda,\eta}(\rho_t)$ being non-increasing, \cite[Proposition 1.6.37]{Tao2011} asserts that 
       $
         J_{\lambda, \eta}(\rho(t)) \le  J_{\lambda, \eta}(\rho_{0}) +  \int_{0}^{t} \frac{{d}}{dr} J_{\lambda, \eta}(\rho(r)) \, dr 
    $ 
    for all $t\ge 0$. We therefore obtain \eqref{J<=J-}. 
\end{proof}

Now, we are ready to present the main convergence result of this paper. 

\begin{theorem}\label{thm:convergence to rho_infty}
Fix $\lambda\in(0,e/2)$, $\eta>0$, and $m>1$. Suppose that $\Psi\in C^1(\R^d)$ satisfies \eqref{>=log growth'} and $\nabla\Psi$ is Lipschitz and fulfills \eqref{Psi' linear growth}. For any $\rho_0\in \mathcal D_2(\R^d)$ such that $J_{\lambda,\eta}(\rho_0) <\infty$, let $\rho:[0,\infty)\to\Dc_2(\R^d)$ be the solution to \eqref{FPE mixed} obtained in Theorem~\ref{thm:sol. to FPE}. 
Then, there exists $\{t_n\}_{n\in\N}$ in $[0,\infty)$ with $t_n\uparrow \infty$ such that 
\begin{equation*}
\rho({t_n})\Lc^d \to \rho_\infty \Lc^d\quad \hbox{weakly in $\Pc(\R^d)$},\ \ \hbox{as $n\to\infty$},
\end{equation*}
where $\rho_\infty\in\Dc_2(\R^d)$, defined in \eqref{rho_infty}, is the unique minimizer of $J_{\lambda,\eta}$ in \eqref{J}. 
\end{theorem}

\begin{proof}
By the definition of $\rho_\infty$ in \eqref{rho_infty} and Lemma~\ref{lem:U}, $\rho_\infty$ satisfies \eqref{Eq: Equilibrium identity} (with $C=C^*$), i.e.,
\begin{equation*} 
  \Psi(y)= - {\lambda}\ln \rho_\infty(y) - \frac{\eta m}{m-1} \rho_\infty^{m-1}(y) +\Big(\frac{\eta m}{m-1} + C^* \Big),\quad y\in\R^d. 
\end{equation*}
When we plug this into \eqref{J}, a direct calculation shows that for any $\rho\in\Dc_2(\R^d)$, 
\begin{align}\label{J=KL...}
J_{\lambda,\eta}(\rho) = \lambda \KL(\rho\Lc^d \| \rho_\infty\Lc^d) &+ \frac{\eta}{m-1} \int_{\R^d} \rho^m dy- \frac{\eta m}{m-1} \int_{\R^d} \rho^{m-1}_{\infty} \rho\, dy +\frac{\eta m}{m-1} +C^*,
\end{align}
where $\KL(\rho\Lc^d \| \theta\Lc^d) := \int_{\R^d} \rho(y) \ln(\rho(y)/\theta(y)) dy$ is the 
Kullback-Leibler divergence of $\rho\Lc^d$ from $\theta\Lc^d$, for any $\rho,\theta\in\Dc(\R^d)$ satisfying ``$\theta=0$ implies $\rho=0$''. Also, as $\Psi\ge 0$, we deduce from \eqref{rho_infty} and \eqref{U formula} that $\rho_\infty$ is bounded. Hence, by taking $\rho = \rho(t) = \rho(t,\cdot)$ in \eqref{J=KL...}, we get
\begin{align*}
\KL(\rho(t)\Lc^d \| \rho_\infty\Lc^d) &\le \frac{1}{\lambda} \left( J_{\lambda,\eta}(\rho(t))  + \frac{\eta m}{m-1} \int_{\R^d} \rho^{m-1}_{\infty}(y) \rho(t,y) dy-\frac{\eta m}{m-1} -C^*\right)\\
&\le \frac{1}{\lambda} \left( J_{\lambda,\eta}(\rho_0) + \frac{\eta m}{m-1} \|\rho_\infty\|^{m-1}_{L^\infty(\R^d)} -\frac{\eta m}{m-1} -C^*\right)<\infty,\quad \forall t>0,  
\end{align*}
where the inequality stems from Lemma~\ref{lem:J decreasing} and $\rho_\infty$ being bounded. With $t\mapsto \KL(\rho(t)\Lc^d \| \rho_\infty\Lc^d)$ uniformly bounded from above, $\{\rho(t)\Lc^d\}_{t>0}$ is tight by \cite[Lemma 1.4.3]{Dupuis97}. As $\KL(\rho\Lc^d \| \rho_\infty\Lc^d)\ge 0$ by definition, when we take $\rho = \rho(t) = \rho(t,\cdot)$ in \eqref{J=KL...}, a similar calculation yields
\begin{align*}
\int_{\R^d} \rho^m(t) dy &\le\frac{m-1}{\eta} \left( J_{\lambda,\eta}(\rho_0) + \frac{\eta m}{m-1} \|\rho_\infty\|^{m-1}_{L^\infty(\R^d)} -\frac{\eta m}{m-1} -C^*\right)<\infty,\quad \forall t>0.  
\end{align*}
That is, $\{\rho(t)\}_{t>0}$ is bounded in $L^m(\R^d)$. On the other hand, recall $v_\rho(t,y)$ in \eqref{v}. Thanks to \eqref{J<=J-} and the fact that $\rho_\infty$ minimizes $J_{\lambda,\eta}$ (by Theorem~\ref{thm:minimizer of J}),
\[
\int_{0}^{t} \|v_\rho(r,\cdot)\|_{L^{2}(\mathbb{R}^{d}; \rho(r)\Lc^d)}^{2} dr \le J_{\lambda, \eta}(\rho_{0}) -J_{\lambda, \eta}(\rho(t))\le J_{\lambda, \eta}(\rho_{0})-J_{\lambda, \eta}(\rho_{\infty})<\infty,\quad \forall t\ge 0, 
\]
Hence, as $t\to\infty$, we get $
\int_{0}^{\infty} \|v_\rho(r,\cdot)\|_{L^{2}(\mathbb{R}^{d}; \rho(r)\Lc^d)}^{2} dr \le J_{\lambda, \eta}(\rho_{0})-J_{\lambda, \eta}(\rho_{\infty})<\infty.
$
This implies that there exists $\{t_n\}_{n\in\N}$ in $[0,\infty)$ with $t_n\uparrow \infty$ such that  
\begin{equation}\label{v L^2 to 0}
\|v_\rho(t_n,\cdot)\|_{L^{2}(\mathbb{R}^{d}; \rho(t_n)\Lc^d)}^{2}\to 0\quad \hbox{as $n\to\infty$}.
\end{equation}
As $\{\rho(t_n)\Lc^d\}_{n\in\N}$ is tight, there is a subsequence (without relabeling) such that $\rho(t_n)\Lc^d\to \mu_*$ weakly for some $\mu_*\in\Pc(\R^d)$, i.e., 
$
\int_{\R^d} \phi(y) \rho(t_n,y) dy \to \int_{\R^d} \phi(y) d\mu_*(y)$ for all $\phi\in C^\infty_c(\R^d). 
$
Also, with $\{\rho(t_n)\}_{n\in\N}$ bounded in $L^m(\R^d)$, there is a further subsequence (without relabeling) such that $\rho(t_n)$ converges weakly in $L^m(\R^d)$, i.e., there exists $\rho_*\in L^m(\R^d)$ such that
$\int_{\R^d} \phi(y) \rho(t_n,y) dy \to \int_{\R^d} \phi(y) \rho_*(y) dy$ for all $\phi\in L^{\frac{m}{m-1}}(\R^d). $
We thus conclude 
$\mu_* = \rho_*\Lc^d$ with $\rho_*\in \Dc(\R^d)\cap L^m (\R^d)$.


With ``$\rho(t_n)\Lc^d \to \mu_* = \rho_*\Lc^d$ weakly in $\Pc(\R^d)$'' established, it remains to show $\rho_*=\rho_\infty$ $\Lc^d$-a.e. To this end, we need to prove several auxiliary results. Recall from Theorem~\ref{thm:sol. to FPE} that $L_F(\rho(t)) = \lambda \rho(t) +\eta\rho^m(t)\in W^{1,1}_{\operatorname{loc}}(\R^d)$ for all $t>0$ (as part of Definition~\ref{def:sol. to FPE}). We claim that $\{L_F(\rho(t_n))\}_{n\in\N}$ is bounded in $W^{1,1}(S)$ for any compact $S\subset \R^d$. As $\| L_F(\rho(t_n))\|_{L^1(\R^d)} \le \lambda +\eta \|\rho(t_n)\|^m_{L^m(\R^d)}$ and $\{\rho(t_n)\}_{n\in\N}$ is bounded in $L^m(\R^d)$, $\{L_F(\rho(t_n))\}_{n\in\N}$ is bounded in $L^1(\R^d)$. Also, by \eqref{nabla L_F} and \eqref{v}, $\nabla L_F(\rho(t_n)) = -(v_\rho(t_n,\cdot) + \nabla\Psi)\rho(t_n)$. Hence, on any compact $S\subset\R^d$, 
\begin{align*}
\int_S |\nabla L_F(\rho(t_n))| dy &\le \int_S |v_\rho(t_n,\cdot)|\rho(t_n) dy + M\int_S \rho(t_n) dy\le \int_{\R^d} |v_\rho(t_n,\cdot)|\rho(t_n) dy + M \int_{\R^d} \rho(t_n) dy\\
&\le \int_{\{|v_\rho(t_n,\cdot)|\le 1\}} |v_\rho(t_n,\cdot)|\rho(t_n) dy+\int_{\{|v_\rho(t_n,\cdot)|>1\}} |v_\rho(t_n,\cdot)|\rho(t_n) dy + M\\
&\le 1+ \|v_\rho(t_n,\cdot)\|^2_{L^2(\R^d,\rho(t_n)\Lc^d)} +M,\qquad \hbox{with}\ M:= \max_{y\in S}|\nabla\Psi(y)|<\infty. 
\end{align*}
As $\sup_{n\in\N}\|v_\rho(t_n,\cdot)\|^2_{L^2(\R^d,\rho(t_n)\Lc^d)}<\infty$ (due to \eqref{v L^2 to 0}), we conclude that $\{L_F(\rho(t_n))\}_{n\in\N}$ is bounded in $W^{1,1}(S)$, as desired. It follows that $\{L_F(\rho(t_n))\}_{n\in\N}$ is also bounded in $BV(S)$, the set of functions on $S$ of bounded variation, as the $BV(S)$-norm of any $f\in W^{1,1}(S)$ coincides with its $W^{1,1}(S)$-norm. Hence, by \cite[Theorem 3.23]{Ambrosio-book-00}, there is a further subsequence (without relabeling) such that $L_{F}(\rho(t_n)) \to f$ in $L_{\operatorname{loc}}^{1}(\mathbb{R}^{d})$ and pointwise $\Lc^d$-a.e., for some $f\in BV_{\operatorname{loc}}(\R^d)$. The same argument below \cite[(10.4.65)]{Ambrosio08} then shows $f = L_F(\rho_*)$ $\Lc^d$-a.e. With $L_F(\rho_*)\in BV_{\operatorname{loc}}(\R^d)$, its distributional derivative (which is a finite $\R^d$-valued measure on $\R^d$) exists. By the $L^2$ duality theory (following the proof of \cite[Theorem 10.4.6]{Ambrosio08}), this distributional derivative 
takes the form $w\rho_*\Lc^d$ for some $\R^d$-valued $w\in L^2(\R^d,\rho_*\Lc^d)$. Hence, $L_F(\rho_*)\in W^{1,1}_{\operatorname{loc}}(\R^d)$ with $\nabla L_F(\rho_*) = w \rho_*$. This, along with $L_{F}(\rho(t_n)) \to L_F(\rho_*)$ in $L_{\operatorname{loc}}^{1}(\mathbb{R}^{d})$, implies $\nabla L_F(\rho(t_n)) \to \nabla L_F(\rho_*)$ in the sense of distributions. Also, by Lemma~\ref{lem:subdiff.=} (i) and its proof, $\rho_*,\rho_*^m\in W^{1,1}_{\operatorname{loc}}(\R^d)$ and $\rho_*$ fulfills \eqref{nabla L_F}.

As $\sup_{n\in\N} \|v_\rho(t_n,\cdot)\|_{L^{2}(\mathbb{R}^{d}; \rho(t_n)\Lc^d)}^{2} <\infty$, we may also apply \cite[Theorem 5.4.4]{Ambrosio08}, which asserts (i) $\gamma_n := (\bm i\times v_{\rho}(t_n,\cdot))_\#(\rho(t_n)\Lc^d)$ has a weak limit $\gamma_*$ in $\Pc(\R^d\times\R^d)$; (ii) $v_\rho(t_n,\cdot)\in L^2(\R^d,\rho(t_n)\Lc^d)$ converges weakly to $v_*\in L^2(\R^d,\rho_*\Lc^d)$ in the sense of \eqref{weak converg. to xi}, where $v_*$ is the barycenter of $\gamma_*$ (see e.g., \cite[Definition 5.4.2]{Ambrosio08}); (iii) for any convex and lower semicontinuous $g:\R^d\to(-\infty,\infty]$, 
\[
\liminf_{n\to\infty} \int_{\R^d}g(v_\rho(t_n,y)) \rho(t_n,y) dy \ge \int_{\R^d} g(v_*(y)) \rho_*(y) dy. 
\]
By taking $g(\cdot) = |\cdot |^2$ in property (iii) and using \eqref{v L^2 to 0}, we obtain
\[
0 = \liminf_{n\to\infty} \|v_\rho(t_n,\cdot)\|_{L^{2}(\mathbb{R}^{d}; \rho(t_n)\Lc^d)}^{2} \ge \|v_*\|_{L^{2}(\mathbb{R}^{d}; \rho_*\Lc^d)}^{2} = \int_{\R^d} |v_*(y)|^2 \rho_*(y)\, dy,
\]
which entails $v_* = 0$ $\rho_*\Lc^d$-a.e. 
Using this and property (ii), we see that for any $\varphi \in C_{c}^{\infty}(\mathbb{R}^{d})$, 
 \begin{equation} \label{Eq: Preliminary to get formula of v*}
        \begin{aligned}
           0= \int_{\R^d} \varphi(y )v_{*}(y) \rho_*(y)\, dy &= \lim_{n \to \infty} \int_{\R^d} \varphi(y) v_\rho(t_n,y) \rho(t_n,y)\, dy \\
            &= -\lim_{n \to\infty} \int_{\R^d} \varphi(y) \Big( \nabla \Psi(y) \rho(t_n,y) + \nabla L_{F}(\rho(t_n,y))\Big) \, dy \\
            &= -\int_{\mathbb{R}^{d}} \varphi(y)  \Big(\nabla \Psi(y) \rho_{*}(y) + \nabla L_{F}(\rho_{*}(y))\Big) \, dy, 
        \end{aligned}        
    \end{equation}
where the third equality follows from \eqref{v} and \eqref{nabla L_F}, the fourth equality stems from $\rho(t_n)\Lc^d\to\rho_*\Lc^d$ weakly in $\Pc(\R^d)$ (note that $\varphi \nabla\Psi$ is bounded and continuous) and $\nabla L_F(\rho(t_n)) \to \nabla L_F(\rho_*)$ in the sense of distributions. 
It follows that
\begin{equation}\label{invariance}
\nabla \Psi(y) \rho_{*}(y) + \nabla L_{F}(\rho_{*}(y)) = 0 \quad \hbox{for $\Lc^d$-a.e.\ $y\in\R^d$}. 
\end{equation}

Now, we claim that $\rho_*>0$ $\Lc^d$-a.e. Thanks to \eqref{invariance}, it can be checked directly that the constant flow $\bar\rho: [0,\infty)\to\Dc(\R^d)$, with $\bar\rho(t) := \rho_*$ for all $t\ge 0$, fulfills \eqref{distri. sense}. As $\rho_*\in \Dc(\R^d)$ and $L_F(\rho_*)\in W^{1,1}_{\operatorname{loc}}(\R^d)$, we conclude that $\bar\rho: [0,\infty)\to\Dc(\R^d)$ is a solution to the Fokker-Planck equation \eqref{FPE mixed} with $\rho_0=\rho_*$ (recall Definition~\ref{def:sol. to FPE}). Furthermore, by $\rho_*\in L^m(\R^d)$ and \eqref{Psi' linear growth}, 
\begin{equation*}
\int_{\R^d} \frac{{ 2\lambda} + { 2\eta} \rho_*^{m-1}(y) + |y \cdot\nabla\Psi(y)|}{(1+|y|)^2} \rho_*(y) dy <\infty. 
\end{equation*}
Hence, we may apply the Figalli-Trevisan superposition principle (i.e., \cite[Theorem 1.1]{BRS21}) as in the proof of Theorem~\ref{thm:sol. to Y}, to obtain a filtered probability space $(\Omega,\cF, \{\cF_t\}_{t\ge 0} ,\P)$ where a solution to 
\begin{equation} \label{SDE mixed u^*}
    dY^*_{t} = - \nabla \Psi(Y^*_{t}) \, dt + \sqrt{{2\lambda} + {2\eta} \rho_*^{m-1}(Y^*_{t})} \, dB_{t}, \quad \rho^{Y_{0}} = \rho_* \in \mathcal{D}(\mathbb{R}^{d}),
\end{equation}
exists and it satisfies $\rho^{Y^*_t} = \rho_*$ for all $t\ge 0$. Consider the strictly positive processes $\sigma_t := \sqrt{{2\lambda} + {2\eta} \rho_*^{m-1}(Y^*_{t})}\ge \sqrt{2\lambda}>0$ and 
$
Z_t := \exp\big(-\int_0^t \frac{\nabla\Psi(Y^*_r)}{\sigma_s} dB_s - \frac12 \int_0^t  \frac{|\nabla\Psi(Y^*_r)|^2}{\sigma^2_r}  dr\big)$ for $t\ge 0.
$
For any $T>0$, \cite[Lemma 4.1.1]{Bensoussan-book-92} asserts $\E[Z_T] =0$, such that $Z$ is a martingale on $[0,T]$. It follows that $\Q(A):= \E[1_A Z_T]$, $A\in\mathcal B(\R^d)$, is a probability measure equivalent to $\P$ and the dynamics of $Y^*$ under $\Q$ is $dY^*_t = \sigma_t \, dB^\Q_{t}$, $t\in[0,T]$, where $B^\Q$ is a $\Q$-martingale. This implies $\Q(Y^*_T\in A)>0$ for all $A\in\mathcal B(\R^d)$ with $\Lc^d(A)>0$. For $A:=\{\rho_* =0\}$, we have $\P(Y^*_T\in A) = \int_{A}\rho_*(y) dy =0$. As $\Q$ is equivalent to $\P$, $\Q(Y^*_T\in A)=0$ must hold. This entails $\Lc^d(A) =0$, i.e., $\rho_* >0$ $\Lc^d$-a.e.

With $\rho_* >0$ $\Lc^d$-a.e., \eqref{invariance} implies $\nabla \Psi + \nabla L_{F}(\rho_{*})/\rho_{*} = 0$ $\Lc^d$-a.e. By \eqref{nabla L_F}, this means $\nabla \Psi+ \lambda \nabla \rho_{*}/\rho_* + \eta m \rho^{m-2}_{*} \nabla \rho_{*} = 0$ $\Lc^d$-a.e., or equivalently,
$
\nabla \big(\Psi +\lambda \ln\rho_* +\frac{\eta m}{m-1} \rho_*^{m-1} \big) =0\ \hbox{$\Lc^d$-a.e.} 
$
Thus, there exists $\bar C\in\R$ such that $\Psi(y) +\lambda \ln\rho_*(y) +\frac{\eta m}{m-1} \rho_*^{m-1}(y) = \bar C$ for $\Lc^d$-a.e.\ $y\in\R^d$. That is, $\rho_*$ fulfills \eqref{Eq: Equilibrium identity}, with $C = \bar C - \frac{\eta m}{m-1}$, for $\Lc^d$-a.e.\ $y\in\R^d$. The same argument in the proof of Lemma~\ref{lem:U} then implies $\rho_*(y) = U(y;C)$ for $\Lc^d$-a.e.\ $y\in\R^d$. As $\rho_*\in\Dc(\R^d)$, we must have $C=C^*$, as specified in Corollary~\ref{coro:rho_infty}. We therefore conclude $\rho_* = \rho_\infty$ $\Lc^d$-a.e.
\end{proof}

The ultimate main result is the following, resulting from Theorem~\ref{thm:convergence to rho_infty}, Proposition~\ref{prop:L^m norm bdd}, and Theorem~\ref{thm:uniqueness Y}. 

\begin{theorem}\label{thm:ultimate result}
Fix $\lambda\in(0,e/2)$, $\eta>0$, and $m>1$. Suppose that $\Psi\in C^2(\R^d)$ satisfies \eqref{Psi' linear growth}, \eqref{Psi increasing}, \eqref{>=log growth'}, $(\Delta \Psi)^+ \in L^{\infty}(\mathbb{R}^{d})$, and $\nabla\Psi$ is Lipschitz. For any $\rho_0\in \mathcal D_2(\R^d)$ such that $J_{\lambda,\eta}(\rho_0) <\infty$, there exists a solution $Y$ to \eqref{SDE mixed} such that $\theta(t) := \rho^{Y_t}\in\Dc_2(\R^d)$ fulfills \eqref{C condition} and \eqref{esssup L^m bdd}. Moreover, for any such a solution $Y$, there exists $\{t_n\}_{n\in\N}$ in $[0,\infty)$ with $t_n\uparrow \infty$ such that 
\[
\theta({t_n})\Lc^d \to \rho_\infty \Lc^d\quad \hbox{weakly in $\Pc(\R^d)$},\ \ \hbox{as $n\to\infty$},
\]
\end{theorem}


\section{A Numerical Example}\label{sec:example}
We consider the double-well function $\Psi:\R\to\R$ introduced in \cite[Section 4]{GXZ22}, i.e., 
\begin{equation*}
\Psi(y) := \begin{cases}
            -12 y - 52, &\ \ y \in (-\infty,-6], \\
            2(y+3)^{2} + 2, &\ \ y \in (-6, 2], \\
            8-y^{2}, &\ \ y \in (2, 6], \\
            (y-4)^{2}, &\ \ y \in (2, 6], \\
            4y - 20, &\ \ y \in (6, \infty).           
       \end{cases}
       \qquad
\vcenter{\hbox{\begin{minipage}{7.3cm}
\centering
\includegraphics[width=7.3cm]{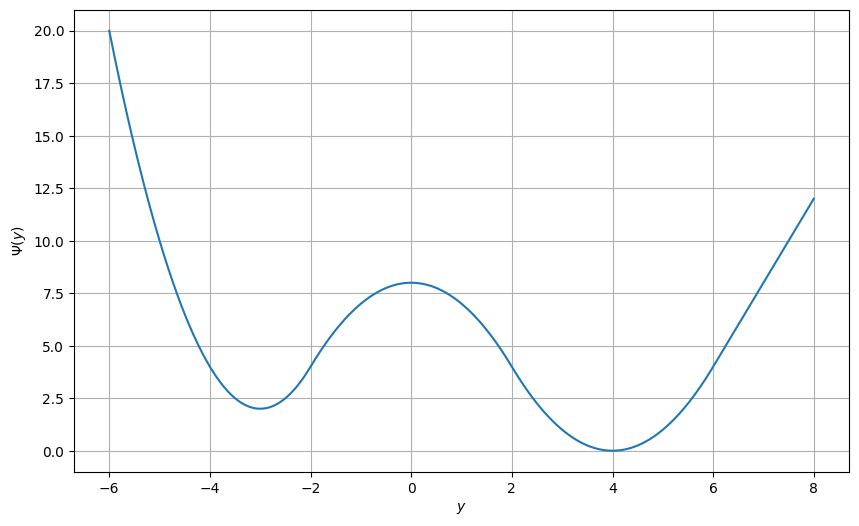}
\end{minipage}}}
\end{equation*}
To minimize $\Psi$, we will simulate the Langevin SDE \eqref{Langevin} and our proposed SDE \eqref{SDE mixed}. 
Let $\Delta t>0$ denote the time step and $\{ \xi_{i} \}_{i\in\N}$ be a sequence of independent and identically distributed standard normal random variables. The Euler-Maruyama schemes for \eqref{Langevin} and \eqref{SDE mixed} are, respectively, 
\begin{equation} \label{LD E-M Process}
    Y_{i+1} := Y_{i} - \Delta t \nabla \Psi(Y_{i}) + \sqrt{2 \lambda \Delta t} \xi_{i},
\end{equation}
and
\begin{equation} \label{Mixed E-M Process}
    Y_{i+1} := Y_{i} - \Delta t \nabla \Psi(Y_{i}) + \sqrt{(2 \lambda + 2 \eta (\rho^{Y_{i}}(Y_{i}))^{m-1}) \Delta t} \xi_{i}. 
\end{equation}
Similarly to \cite[Section 4]{GXZ22}, we initiate \eqref{LD E-M Process} and \eqref{Mixed E-M Process} at $y=-3$, the suboptimal local minimum of $\Psi$, perform the simulation up to $i=500$, and repeat the procedure 500 times. Note that in \eqref{Mixed E-M Process}, the density $\rho^{Y_i}$ is approximated by kernel density estimation. The performance of the algorithms is evaluated by $\{\E[\Psi(Y_i)]\}_{i=1}^{500}$, approximated by the sample average of the 500 simulated paths. 

Figure~\ref{fig: numerical results} presents the result: the blue curve is obtained from \eqref{LD E-M Process} with $\lambda = 500*(0.5)^{10} \approx 0.4883$ and $\Delta t = 0.5$, which are the best performing parameters reported in \cite[Section 4]{GXZ22}; other curves are obtained from \eqref{Mixed E-M Process} with $\lambda$ and $\Delta t$ as above, $\eta = 500*(0.5)^{9} \approx 0.9766$, and $m = 2, 5, 8, 10$. Clearly, $\E[\Psi(Y_i)]$ decays much faster under our algorithm \eqref{Mixed E-M Process} (for any $m$ value) than under the Langevin algorithm \eqref{LD E-M Process}. Moreover, as $m$ increases, the decay of  $\E[\Psi(Y_i)]$ accelerates and it eventually converges to a lower value (particularly, lower than that achieved by \eqref{LD E-M Process}). 

Under \eqref{LD E-M Process}, there is a tradeoff between the convergence rate and quality of the long-time limit. With a smaller $\lambda$, as the limiting distribution $\rho_\lambda$ in \eqref{rholambda} concentrates more on the global minimizer of $\Psi$, the long-time limit of $\E[\Psi(Y_i)]$ is generally smaller. It may yet take much longer to achieve this better (smaller) limit, as the convergence rate is exponential in $1/\lambda$; see \cite{BEGK04, BGK05}. 
Intriguingly, 
Figure~\ref{fig: numerical results} suggests that under \eqref{Mixed E-M Process}, one can {\it simultaneously} improve the convergence rate and quality of the long-time limit, by tuning location-wise random perturbation: a larger $m$ makes $(\rho^{Y_{i}}(\cdot))^{m-1}$ more sensitive to the landscape imposed by $\Psi$, rendering the simultaneous improvement. 
\begin{figure}
    \centering
    \includegraphics[width=0.6\linewidth]{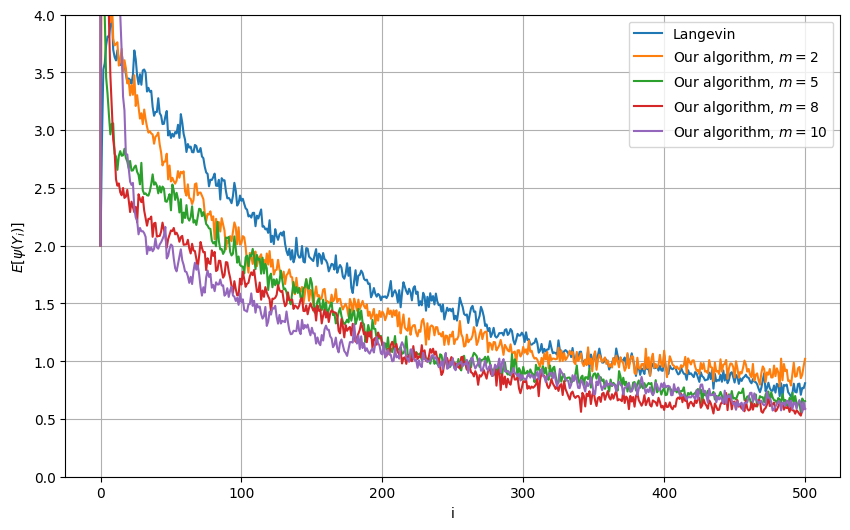}
    \caption{Performance of the tested algorithms.}
    \label{fig: numerical results}
\end{figure}





\appendix

\section{Proof of Lemma~\ref{lem:doubling}}\label{sec:proof of lem:doubling}
First, fix any $\lambda\in (0,1/\ln 2]$. Note that ``$H$ satisfies \eqref{doubling} for some $C>0$'' is equivalent to
\begin{equation}\label{G(x,y)}
\exists C>0\ \ \hbox{s.t.}\ \ G(x,y) := \frac{(x+y)^{\lambda(x+y)}}{x^{\lambda C x} y^{\lambda C y}} \le e^C\quad \forall x,y\ge 0.
\end{equation}
Let us first find the global maximum of $G(x,y)$. The first-order condition, i.e., 
$G_x(x,y) = \lambda G(x,y) \big(\ln(x+y)-C\ln x +1-C\big) = 0$ and 
$G_y(x,y) = \lambda G(x,y) \big(\ln(x+y)-C\ln y +1-C\big) = 0$,
reduces to $e^{C-1} x^C = e^{C-1} y^C = x+y$, which has a unique solution $x^*=y^* = 2^{\frac{1}{C-1}} e^{-1}$. Then, direct calculations yield
$
G_{xx}(x^*,y^*) = G_{yy}(x^*,y^*) = \big(\frac12-C\big)\frac{\lambda}{x^*} G(x^*,x^*)$ and $G_{xy}(x^*,y^*) = \frac{\lambda}{2x^*} G(x^*,x^*). 
$
Let $\mathcal H(x^*,y^*)$ denote the Hessian matrix of $G$ evaluated at $(x^*,y^*)$. Observe that for any $C>1$, we have $G_{xx}(x^*,y^*)<0$ and 
$
\hbox{det}(\mathcal H(x^*,y^*)) = G_{xx}(x^*,y^*) G_{yy}(x^*,y^*) - (G_{xy}(x^*,y^*))^2 = -\big(\frac{\lambda}{x^*}\big)^2 G^2(x^*,x^*) (1-C) C>0,
$
i.e., $\mathcal H(x^*,y^*)$ is negative definite. This implies that for any $C>1$, $(x^*,y^*)$ is the unique maximizer of $G(x,y)$. Thus, 
$G(x,y) \le G(x^*,y^*) = \exp\big({(C-1)\lambda 2^{\frac{C}{C-1}}e^{-1}}\big)$ for all $x,y\ge 0$. 
Hence, to prove \eqref{G(x,y)}, it suffices to choose $C>1$ such that $(C-1)\lambda 2^{\frac{C}{C-1}}e^{-1}\le C$, or equivalently, 
$
\lambda \le e f\big(\frac{C}{C-1}\big)$ with $f(z) := z/{2^{z}}$ for $z\ge 0.
$
Note that $f(z)$ achieves its maximum $e^{-1}/\ln 2$ at $z^*= 1/\ln 2 >1$. Hence, by taking $C=\frac{1}{1-\ln 2}>1$, the desired relation becomes $\lambda\le e\big(\frac{C}{C-1}\big) = e f(z^*) = 1/\ln 2$, which holds due to our choice $\lambda\in (0,1/\ln 2]$. 

Now, fix any $\lambda\in (0,e/2]$, $\eta>0$, and integer $m \ge 2$. For any $x,y\ge 0$, we first recall the relation 
\begin{equation}\label{(x+y)^m}
(x+y)^m \le (2^m-1)(x^m+y^m). 
\end{equation}
Take $C_1 = \frac{1}{1-\ln 2}$ and $C_2 = 2^m-1$. By the definition of $F$ in \eqref{F}, 
\begin{align*}
F&(x+y) = H(x+y) +\frac{\eta}{m-1} (x+y)^m \le C_1(1+H(x)+H(y)) + C_2 \bigg(\frac{\eta}{m-1} x^m + \frac{\eta}{m-1} y^m\bigg)\\
&\le (C_1\vee C_2) \left(1+H(x)+H(y) + \frac{\eta}{m-1} x^m + \frac{\eta}{m-1} y^m\right) = (C_1\vee C_2) (1+ F(x)+F(y)), 
\end{align*}
where the first inequality follows from ``$H$ satisfies \eqref{doubling} with $C=C_1$, for any $\lambda\in (0,1/\ln 2]$'' (Here, $\lambda\le e/2 <\ln 2$) and \eqref{(x+y)^m}, and the second inequality holds relying partly on the nonnegativity of $1+H(x)+H(y)$ (Indeed, as $\lambda \le e/2$, $1+H(x)+H(y)= 1 + \lambda x\ln x +\lambda y\ln y\ge 1-2\lambda/e\ge 0$). 

\small{
\bibliographystyle{siam}
\bibliography{bibliography}
}
\end{document}